\documentclass[a4paper, 10pt]{article}

\usepackage[english]{babel}
\usepackage[utf8]{inputenc}

\usepackage[T1]{fontenc}
\usepackage{lmodern}

\usepackage{amsmath, amssymb, amsfonts, mathtools, amsthm}
\usepackage{url}
\usepackage{authblk}

\usepackage{tikz}
\usetikzlibrary{arrows}
\usetikzlibrary{patterns}
\usetikzlibrary{positioning}
\usetikzlibrary{shapes.misc}

\newtheorem{proposition}{Proposition}
\newtheorem{lemma}{Lemma}

\providecommand{\abs}[1]{\lvert#1\rvert}
\providecommand{\norm}[1]{\lVert#1\rVert}

\newcommand{\R}{\mathbb{R}}

\newcommand{\x}{\mathbf{x}}
\newcommand{\y}{\mathbf{y}}

\begin{document}

\title{Optical Flow on Evolving Sphere-Like Surfaces}

\author[1]{Lukas F. Lang}
\author[1,2]{Otmar Scherzer}
\affil[1]{\footnotesize Computational Science Center, University of Vienna, Oskar-Morgenstern-Platz\ 1, 1090 Vienna, Austria}
\affil[2]{Radon Institute of Computational and Applied Mathematics, Austrian Academy of Sciences, Altenberger Str.\ 69, 4040 Linz, Austria}

\date{}
\maketitle

\begin{abstract}
\noindent
In this work we consider optical flow on evolving Riemannian 2-manifolds which can be parametrised from the 2-sphere.
Our main motivation is to estimate cell motion in time-lapse volumetric microscopy images depicting fluorescently labelled cells of a live zebrafish embryo.
We exploit the fact that the recorded cells float on the surface of the embryo and allow for the extraction of an image sequence together with a sphere-like surface.
We solve the resulting variational problem by means of a Galerkin method based on vector spherical harmonics and present numerical results computed from the aforementioned microscopy data.
\end{abstract}

\section{Introduction}

Motion estimation is a fundamental problem in image analysis and computer vision.
An important task within is optical flow computation.
It is concerned with the inference of a vector field describing the displacements of brightness patterns, such as moving objects, in a sequence of images.
Ever since the seminal work of Horn and Schunck~\cite{HorSchu81} a variety of reliable and efficient methods have been proposed and successfully applied in a wide number of fields.

Primarily, optical flow is computed in the plane.
However, it is readily generalised to non-Euclidean settings allowing, for instance, for cell motion analysis in time-lapse microscopy data.
It has been only recently that high-resolution observations of biological model organisms such as the zebrafish became possible.
Despite its importance for tissue and organ formation, little is known about cell migration and proliferation patterns during the zebrafish's early embryonic development~\cite{AmaLemMosMcDWan14, SchmShaScheWebThi13}.
Fluorescence microscopy nowadays allows to record time-lapse images on the scale of single cells, see e.g.~\cite{Kel13, MegFra03, SchmShaScheWebThi13}.
Increasing spatial as well as temporal resolutions result in vast amounts of data, rendering extraction of information through visual inspection carried out by humans impracticable.
Automated cell motion estimation therefore is key to large-scale analysis of such data.
Optical flow computation delivers necessary quantitative methods and leads to insights into the underlying cellular mechanisms and the dynamic behaviour of cells.
See, for example,~\cite{AmaMyeKel13, MelCamLomRizVer07, QueMenCam10, SchmShaScheWebThi13} and the references therein.

The primary biological motivation for this work is the desire to analyse cell motion in a living zebrafish during early embryogenesis.
The data at hand depict endodermal cells expressing a green fluorescent protein.
By virtue of laser-scanning microscopy, (volumetric time-lapse) 4D images of these labelled cells can be recorded without capturing the background.
It is known that endodermal cells float on a so called \emph{monolayer} during early embryonic development meaning that they do not stack on top of each other~\cite{WarNus99}.
Figure~\ref{fig:raw} depicts two frames of the captured sequence, containing only the upper hemisphere of the animal embryo.
Observe the salient formation of the cells and the noise present in the images.
More precisely, one can see the nuclei of cells forming a round surface in a single layer.
For more details on the microscopy data we refer to Sec.~\ref{sec:data}.

We exploit this situation and model this layer as an evolving surface.
A natural candidate for a parametrisation of such a zebrafish embryo is a \emph{sphere-like surface}.
It is topologically diffeomorphic to the 2-sphere $\mathcal{S}^2$ and most commonly defined as the set of points
\begin{equation*}
	\{ \tilde{\rho}(x) x: x \in \mathcal{S}^2 \}.
\end{equation*}
The function $\tilde{\rho}: \mathcal{S}^2 \to (0, \infty)$ can be thought of as a radial deformation of $\mathcal{S}^2$ and will have a dependence on time in the present paper.
As a consequence, changes in the embryo's geometry are attributed accordingly, albeit valid only during early stages of its development as cells tend to cluster subsequently.
The main intention of this work is to conceive cell motion only on this moving 2-dimensional manifold.
As a result we are able to reduce the spatial dimension of the data allowing for more efficient motion estimation in microscopy data.
Figure~\ref{fig:data} depicts two frames of the surface together with images obtained by restriction of the volumetric microscopy data in Fig.~\ref{fig:raw}.

\begin{figure}[t]
	\includegraphics[width=0.49\textwidth]{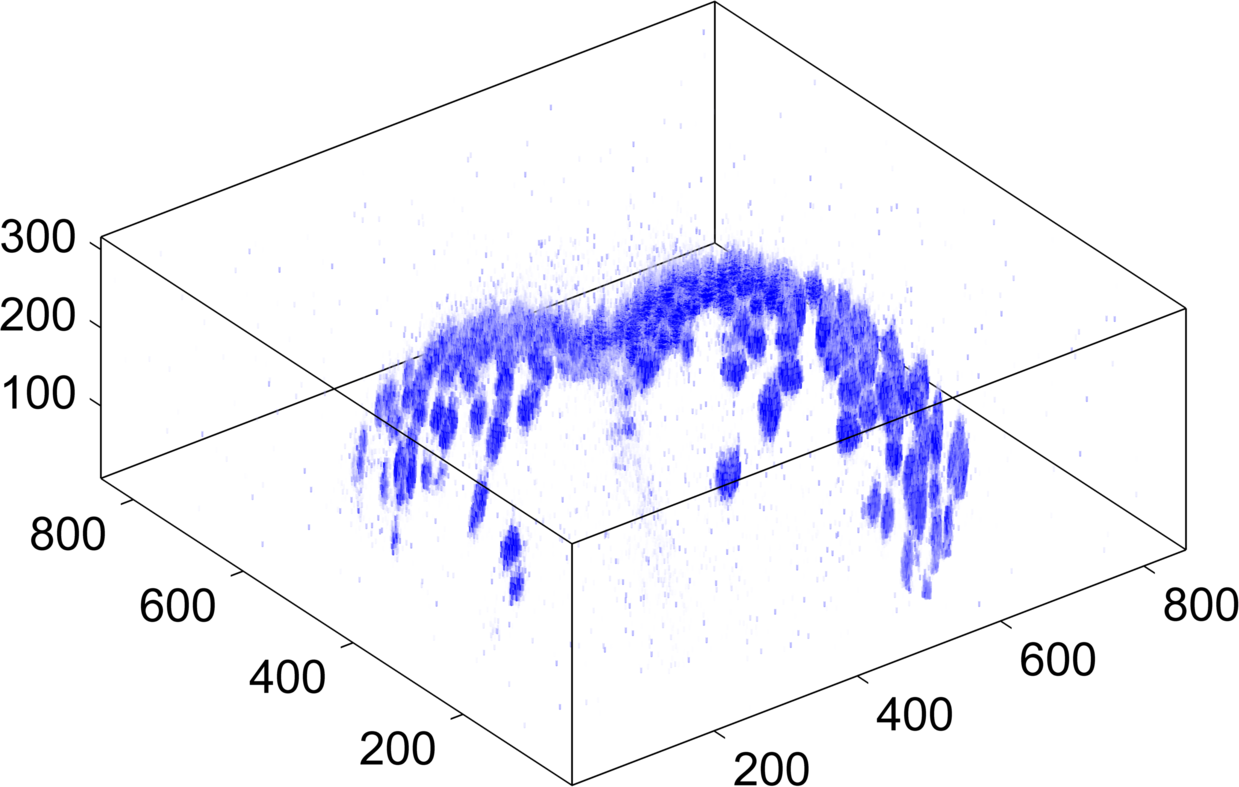} \hfill
	\includegraphics[width=0.49\textwidth]{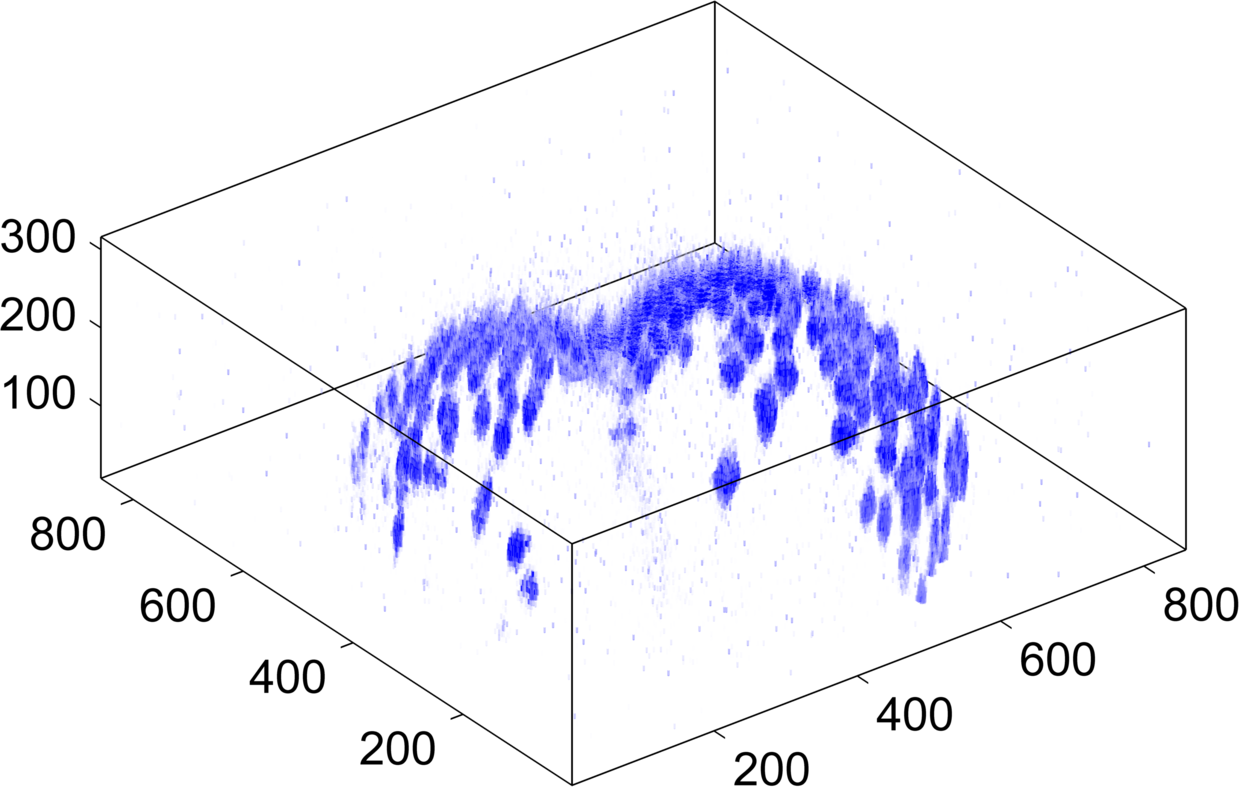}
	\caption{Frames 70 (left) and 71 (right) of the volumetric zebrafish microscopy images recorded during early embryogenesis. The sequence contains a total number of 75 frames. Fluorescence response is indicated by blue colour and is proportional to the observed intensity. All dimensions are in micrometer ($\mu$m).}
	\label{fig:raw}
\end{figure}

\begin{figure}[t]
	\includegraphics[width=0.49\textwidth]{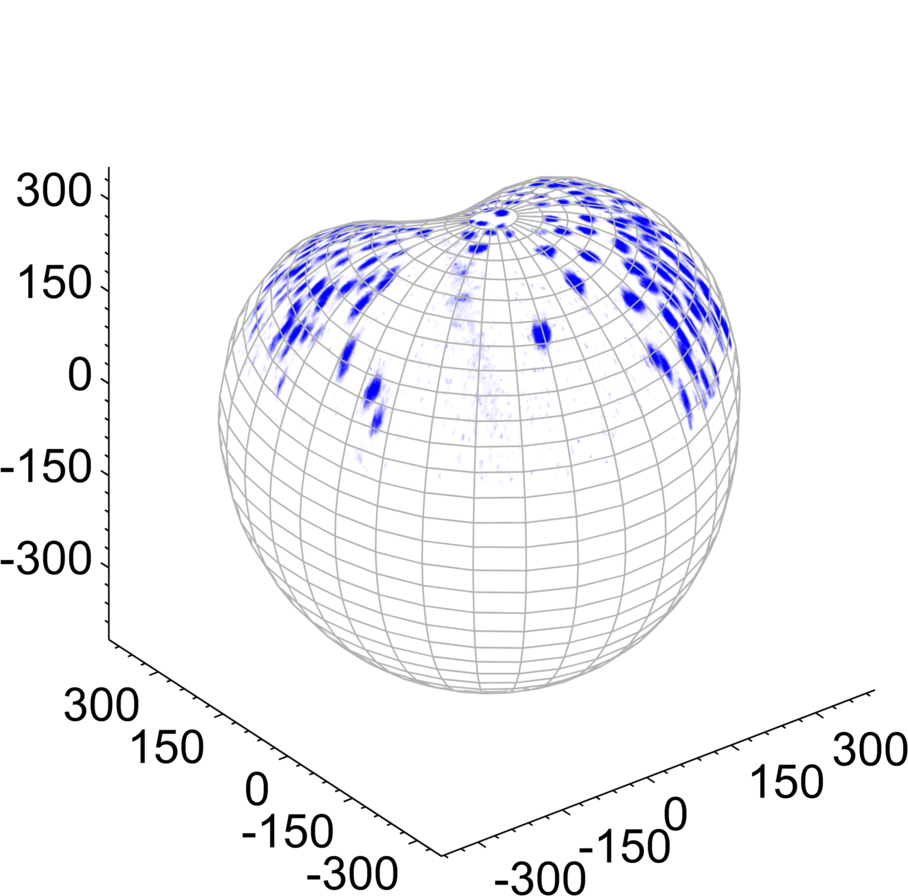} \hfill
	\includegraphics[width=0.49\textwidth]{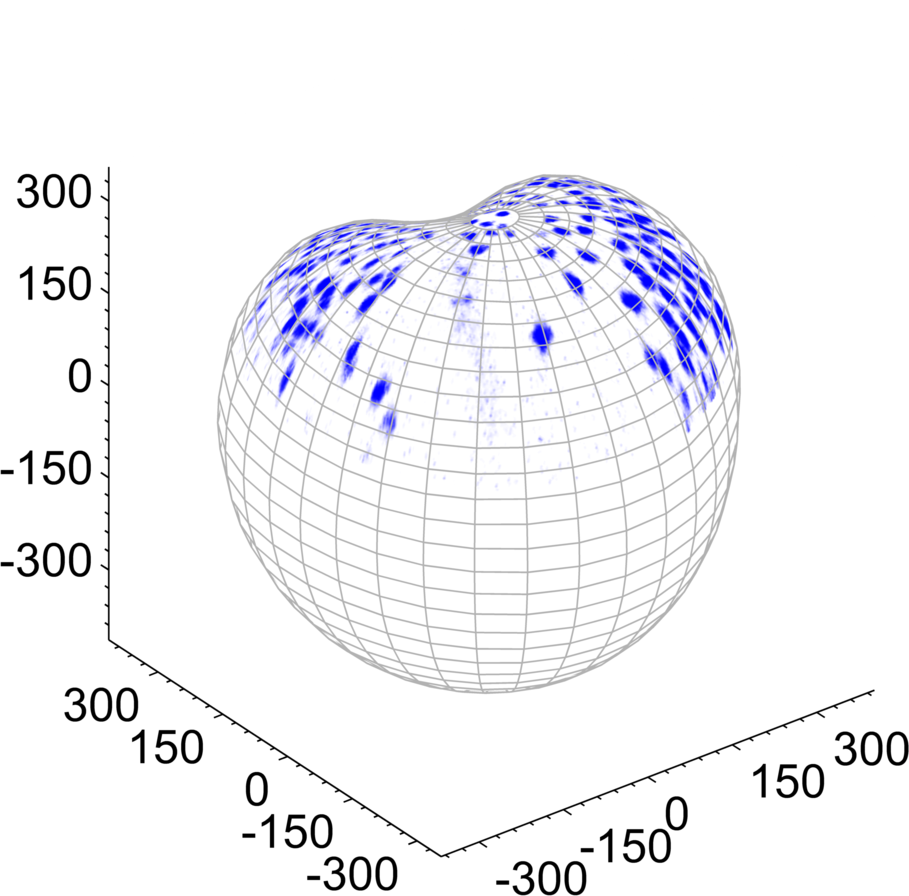} \\
	\includegraphics[width=0.49\textwidth]{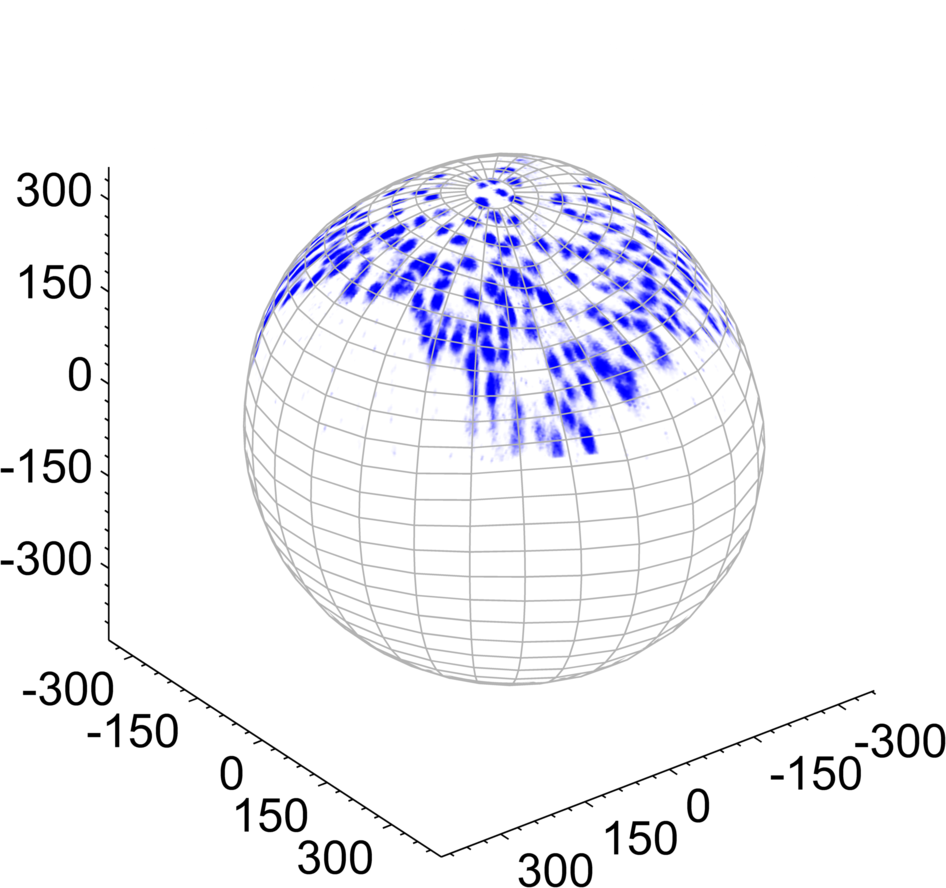} \hfill
	\includegraphics[width=0.49\textwidth]{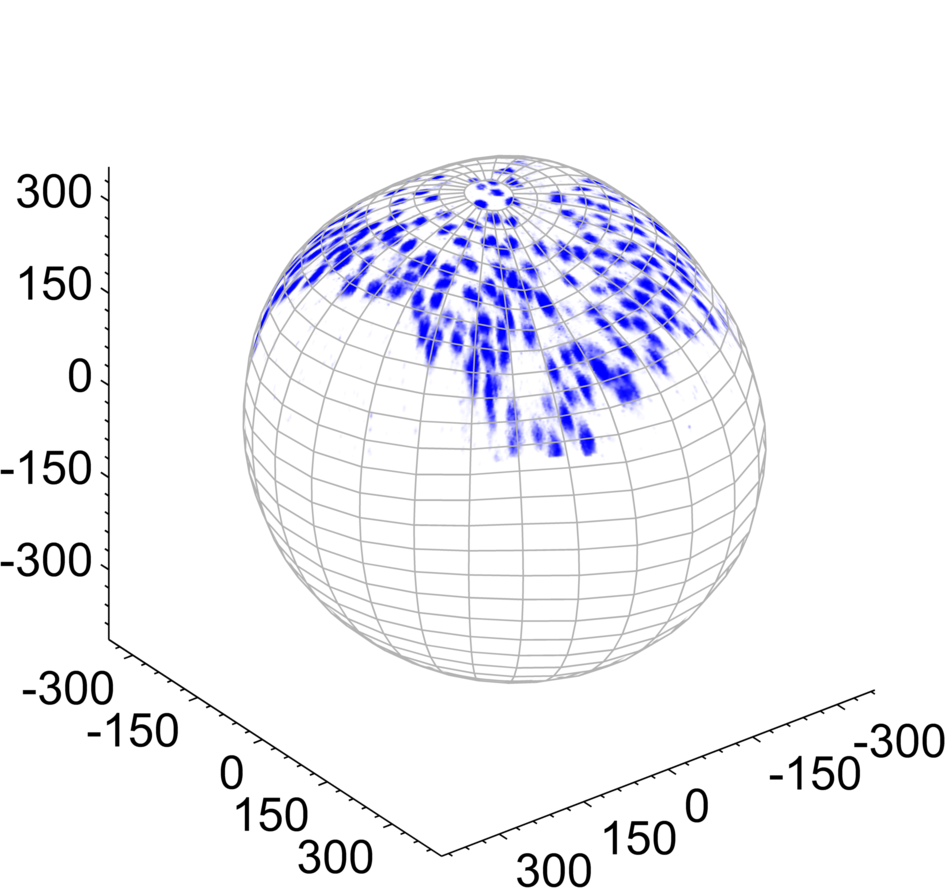} 
	\caption{Depicted are frames no. 70 (left) and 71 (right) of the processed zebrafish microscopy sequence. Top and bottom row differ by a rotation of 180 degrees around the $x_{3}$-axis. All dimensions are in micrometer ($\mu$m).}
	\label{fig:data}
\end{figure}

In this work we model the data as a time-dependent non-negative function $\hat{f}$.
Its value directly corresponds to the fluorescence response of the observed cells.
For a fixed time instant $t \in [0, T]$, the domain of $\hat{f}$ is presumed to be a closed surface $\mathcal{M}_{t} \subset \R^{3}$.
We assume that this surface can be parametrised by a smooth radial map from the 2-sphere.
The temporal evolution of the data $\hat{f}$ can then be tracked by solving an optical flow problem on this moving surface or, more conveniently, an equivalent problem on the round sphere. 

Traditionally, the starting point for optical flow is the assumption of constant brightness: a point moving along a trajectory does not change its intensity over time.
On a moving domain $\mathcal{M} = \{ \mathcal{M}_{t} \}_{t}$ one equivalently seeks, for every time $t \in [0, T]$, a tangent vector field $\mathbf{\hat{v}}$ that solves a generalised optical flow equation
\begin{equation}
	d_{t}^{\mathbf{\hat{V}}} \hat{f} + \nabla_{\mathcal{M}} \hat{f} \cdot \mathbf{\hat{v}} = 0
	\label{eq:gofe}
\end{equation}
at every point $x \in \mathcal{M}$, where $\hat{f}$ is the image sequence living on $\mathcal{M}$.
Here, for a fixed time $t$, $\nabla_{\mathcal{M}}$ denotes the (spatial) surface gradient, dot the standard inner product, and $d_{t}^{\mathbf{\hat{V}}} \hat{f}$ an appropriate temporal derivative.

The optical flow problem is ill-posed meaning that equation \eqref{eq:gofe} is not uniquely solvable.
A common approach to deal with non-uniqueness is Tikhonov regularisation, which consists of computing a minimiser of
\begin{equation*}
	\mathcal{E}_{\alpha}(\mathbf{\hat{v}}) = \mathcal{D}(\mathbf{\hat{v}}, \hat{f}) + \alpha \mathcal{R}(\mathbf{\hat{v}}).
\end{equation*}
The first term of the sum is usually the squared $L^{2}$ norm of the left-hand side of~\eqref{eq:gofe} and, in the present article, the second term will be an $H^{1}$ Sobolev norm.

\subsection{Contributions}

The primary concern of this article is optical flow computation on evolving 2-dimensional Riemannian manifolds which can be parametrised from the sphere.
Motivated by the aforementioned zebrafish microscopy data we consider closed surfaces for which the mapping
\begin{equation}
	(t, x) \mapsto \tilde{\rho}(t, x) x, \quad x \in \mathcal{S}^2
\label{eq:spheremap}
\end{equation}
is a diffeomorphism between the 2-sphere and $\mathcal{M}_{t}$ for every time $t \in [0, T]$.
As a prototypical example we restrict ourselves to radially parametrised surfaces as they suit quite naturally to the given data.

The contributions of this work are as follows.
First, we give a variational formulation of optical flow on 2-dimensional closed Riemannian manifolds.
We assume a dependence on time and speak of evolving surfaces.
The main idea is to solve the problem by a Galerkin method in a finite-dimensional subspace of an appropriate (vectorial) Sobolev space.
We take advantage of the fact that tangential vector spherical harmonics form a complete orthonormal system for $L^{2}(\mathcal{S}^2, T\mathcal{S}^2)$.
The sought vector field is thus uniquely determined when expanded in terms of the pushforward---by means of the differential of~\eqref{eq:spheremap}---of these functions.
From that we arrive at a minimisation problem over $\R^{n}$, where $n$ is the dimension of the finite-dimensional space, and state the optimality conditions.
They can be written purely in terms of spherical quantities and solved on the 2-sphere.
To this end, we use a standard polyhedral approximation and locally interpolate spherical functions by piecewise quadratic polynomials.
For numerical integration we employ appropriate quadrature rules on the approximated sphere.

Second, to obtain the smooth sphere-like surface, which is described by the map~\eqref{eq:spheremap}, from the observed microscopy data, we formulate another variational problem on the sphere.
The problem is essentially surface interpolation with $H^{s}$ Sobolev seminorm regularisation.
Approximate cell centres serve as sample points of the surface.
In particular, our microscopy data are supported only on the upper hemisphere, see Figs.~\ref{fig:raw} and~\ref{fig:data}.
Scalar spherical harmonics are the appropriate choice for the numerical solution of the surface fitting problem, as they provide great flexibility with respect to the chosen space $H^{s}$.

Finally, we present numerical experiments on the basis of the mentioned cell microscopy data of a live zebrafish.
To this end we compute an approximation of the sphere-shaped embryo and obtain a sequence of images living on this moving surface.
Eventually, we solve for the optical flow and present the results in a visually adequate manner.

\subsection{Related Work}

The first variational formulation of optical flow is commonly attributed to Horn and Schunck~\cite{HorSchu81}.
They attempted to compute a displacement field in $\R^{2}$ by minimising a Tikhonov-regularised energy functional.
It favours spatially regular vector fields by penalising its squared $H^{1}$ Sobolev seminorm.
For introductory material on the subject we refer to~\cite{AubDerKor99, AubKor06} and to~\cite{WeiBruBroPap06} for a survey on various optical flow functionals.
Well-posedness of the aforementioned energy was first shown by Schn\"orr~\cite{Schn91a}.
Moreover, there the problem was extended to irregular planar domains and solved by means of finite elements.

Weickert and Schn\"orr~\cite{WeiSchn01b} considered a spatio-temporal model by extending the domain to $\R^{2} \times [0, T]$.
It additionally favours temporal regularity of the solution by including first derivatives with respect to time.
Such models are of particular interest whenever trajectories are to be computed from the optical flow field.
A unifying framework including several spatial as well as temporal regularisers was proposed in~\cite{WeiSchn01a}.
For the purpose of evaluation and flow field visualisation a framework was created by Baker~et~al.~\cite{BakSchaLewRotBla11}.

Recently, generalisations to non-Euclidean domains have gained increasing attention.
In~\cite{ImiSugTorMoc05} and~\cite{TorImiSugMoc05} optical flow was considered in a spherical setting.
Lef\`{e}vre and Baillet~\cite{LefBai08} adapted the Horn-Schunck functional to surfaces embedded in $\R^{3}$.
Following Schn\"{o}rr~\cite{Schn91a}, they proved well-posedness of their formulation and employed a finite element method for solving the discrete problem on a triangle mesh.
With an application to cell motion analysis, Kirisits~et~al.~\cite{KirLanSch13, KirLanSch15} recently considered optical flow on evolving surfaces with boundary.
They generalised the spatio-temporal model in~\cite{WeiSchn01b} to a non-Euclidean and dynamic setting.
Eventually, the problem was tackled numerically by solving the corresponding Euler-Lagrange equations in the coordinate domain.
Similarly, Bauer~et~al.~\cite{BauGraKir15} studied optical flow on time-varying domains, with and without spatial boundary.
They proposed a treatment on surfaces parametrised by product manifolds, constructed an appropriate Riemannian metric, and proved well-posedness of their formulation.

In Kirisits~et~al.~\cite{KirLanSch14}, the authors considered various decomposition models for optical flow on the 2-sphere.
The proposed functionals were solved by means of projection to a finite-dimensional space spanned by vector spherical harmonics.
Concerning projection methods, Schuster and Weickert~\cite{SchuWei07} solved the optical flow problem in $\R^{2}$ solely based on regularisation by discretisation.

Regarding sphere-like surfaces and spherical harmonics expansion of closed surfaces we refer to~\cite{PenDin07} and the references therein.

Finally, let us mention~\cite{AmaMyeKel13, MelCamLomRizVer07, QueMenCam10, SchmShaScheWebThi13}, where optical flow was employed for the analysis of cell motion in microscopy data.
In particular, in Schmid~et~al.~\cite{SchmShaScheWebThi13} the embryo of a zebrafish was modelled as a round sphere and motion of endodermal cells computed in map projections.

The remainder of this article is structured as follows.
In Sec.~\ref{sec:background}, we formally introduce evolving sphere-like surfaces, recall the definition of vectorial Sobolev spaces on manifolds, and discuss both scalar and vector spherical harmonics on the 2-sphere.
Section~\ref{sec:model} is dedicated to optical flow on evolving surfaces and our variational formulation.
In Sec.~\ref{sec:numerics} we discuss the numerical solution.
In particular, we propose to solve the resulting energy in a finite-dimensional subspace and rewrite the optimality conditions to be defined solely on the 2-sphere.
Moreover, we show how to fit a sphere-like surface to the labelled cells in the microscopy data.
Finally, in Sec.~\ref{sec:experiments}, we solve for the optical flow field and visualise the results.
The appendix contains deferred material.
\section{Notation and Background} \label{sec:background}

\subsection{Sphere-Like Surfaces} \label{sec:diffgeo}

Let
\begin{equation*}
	\mathcal{S}^{2} = \{ x \in \R^3 : \norm{x} = 1 \}
\end{equation*}
be the 2-sphere embedded in the 3-dimensional Euclidean space. 
The norm of $\R^{n}$, $n = \{ 2, 3\}$, is denoted by $\norm{x} = \sqrt{x \cdot x}$.
By
\begin{equation}
	\x: \Omega \subset \R^{2} \to \R^{3}
\label{eq:sphereparam}
\end{equation}
we denote a smooth (local) parametrisation of $\mathcal{S}^{2}$ mapping coordinates $\xi = (\xi^{1}, \xi^{2})^{\top} \in \Omega$ to points $x = (x^{1}, x^{2}, x^{3})^{\top} \in \mathcal{S}^{2}$.

Furthermore, let $I \coloneqq [0, T] \subset \R$ denote a time interval and let $\mathcal{M} = \{ \mathcal{M}_{t} \}_{t \in I}$ be a family of closed smooth 2-manifolds $\mathcal{M}_{t} \subset \R^{3}$.
Each $\mathcal{M}_{t}$, $t \in I$, is assumed to be regular and oriented by the outward unit normal field $\mathbf{\hat{N}}(t, x) \in \R^{3}$, $x \in \mathcal{M}_{t}$.
We assume that $\mathcal{M}$ (locally) admits a smooth parametrisation of the form
\begin{equation}
	\y: I \times \Omega \to \R^{3}, \quad (t, \xi^{1}, \xi^{2})^{\top} \mapsto \tilde{\rho}(t, \x(\xi^{1}, \xi^{2})) \x(\xi^{1}, \xi^{2}) \in \mathcal{M}_{t}
\label{eq:param}
\end{equation}
and call $\mathcal{M}$ an \emph{evolving sphere-like surface}.

We denote by $\hat{f}: \mathcal{M} \to \R$ a smooth function on the moving surface.
Its coordinate representation $f: I \times \Omega \to \R$ and its corresponding spherical representation $\tilde{f}: I \times \mathcal{S}^{2} \to \R$ are given by
\begin{equation}
\label{eq:f_on_manifold}
	f(t, \xi) = \tilde{f}(t, \x(\xi)) = \hat{f}(t, \y(t, \xi)).
\end{equation}
As a notational convention we indicate functions living on $\mathcal{S}^{2}$ with a tilde and functions on $\mathcal{M}$ with a hat, respectively.
Their corresponding coordinate version is treated without special indication.

For convenience, we define smooth extensions of $\tilde{f}$ and $\hat{f}$ to $\R^{3} \setminus \{ 0 \}$ by
\begin{equation}
	\tilde{\bar{f}}(t, x) \coloneqq \tilde{f} \left(t, \frac{x}{\norm{x}} \right) \text{ and } \hat{\bar{f}}(t, x) \coloneqq \hat{f} \left(t, \tilde{\rho}\left(t, \frac{x}{\norm{x}} \right) \frac{x}{\norm{x}} \right),
\label{eq:extension}
\end{equation}
respectively.
Note that, while $\tilde{\bar{f}}$ is constant in the direction of the surface normal of $\mathcal{S}^{2}$, the extension $\hat{\bar{f}}$ in general is not.
We point at Fig.~\ref{fig:surfaces} illustrating the setting.

\begin{figure}[t]
	\begin{center}
	\begin{tikzpicture}
		\draw [-stealth', thick, gray] (canvas polar cs:angle=45,radius=1cm) to +(canvas polar cs:angle=45,radius=1cm);
		\draw [-stealth', thick, gray] (canvas polar cs:angle=70,radius=1.77cm) to +(canvas polar cs:angle=30,radius=1cm);
		\draw [thick] (0, 0) circle (1);
		\draw [thick] (0, 2) .. controls (canvas polar cs:angle=70,radius=2.05cm) and (canvas polar cs:angle=65,radius=1.55cm) .. (canvas polar cs:angle=45,radius=1.5cm);
		\draw [thick] (canvas polar cs:angle=45,radius=1.5cm) .. controls (canvas polar cs:angle=30,radius=1.6cm) and (canvas polar cs:angle=20,radius=2.05cm) .. (2, 0);
		\draw [thick] (-2, 0) .. controls (-2, 1.11) and (-1.11, 2) .. (0, 2);
		\draw [thick] (-2, 0) .. controls (-2, -1.11) and (-1.11, -2) .. (0, -2) .. controls (1.11, -2) and (2, -1.11) .. (2, 0);
		\draw (0, 0) to (canvas polar cs:angle=70,radius=3cm);
		\filldraw [black] (canvas polar cs:angle=70,radius=1cm) circle (2pt);
		\filldraw [black] (canvas polar cs:angle=70,radius=1.77cm) circle (2pt);
		\node [left] at (canvas polar cs:angle=73,radius=0.7cm) {$x$};
		\node [left] at (canvas polar cs:angle=73,radius=1.6cm) {$\tilde{\rho}(t, x) x$};
		\node [gray, right] at (canvas polar cs:angle=55,radius=2.6cm) {$\mathbf{\hat{N}}$};
		\node [gray, right] at (canvas polar cs:angle=45,radius=2cm) {$\mathbf{\tilde{N}}$};
		\node [right] at (canvas polar cs:angle=-40,radius=1.05cm) {$\mathcal{S}^2$};
		\node [right] at (canvas polar cs:angle=-40,radius=2.1cm) {$\mathcal{M}_{t}$};
	\end{tikzpicture}
	\end{center}
	\caption{Schematic illustration of a cut through the surfaces $\mathcal{S}^2$ and $\mathcal{M}_{t}$ intersecting the origin. In addition, a radial line along which the extensions $\tilde{\bar{f}}$ and $\hat{\bar{f}}$ are constant is shown. The surface normals are shown in grey.}
	\label{fig:surfaces}
\end{figure}
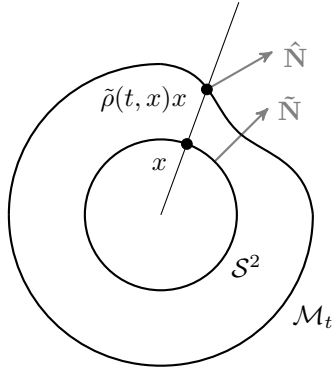

Similarly, for vector-valued functions $\mathbf{\tilde{u}}: I \times \mathcal{S}^{2} \to \R^{3}$ and $\mathbf{\hat{u}}: \mathcal{M} \to \R^{3}$ the extensions to $\R^{3} \setminus \{ 0 \}$ are defined component-wise and for all times $t \in I$.
They are denoted by $\mathbf{\tilde{\bar{u}}}$ and $\mathbf{\hat{\bar{u}}}$, respectively.
As a notational convention, boldface letters are used to denote vector fields.
Moreover, we distinguish between lower and upper case boldface letters.
The former identify tangent vector fields and their extensions to $\R^{3} \setminus \{ 0 \}$ whereas the latter indicate general vector fields in $\R^{3}$.

For a differentiable function $f: I \times \Omega \to \R$, we write $\partial_{i} f$ as an abbreviation for the partial derivative of $f$ with respect to $\xi^{i}$.
That is, $\nabla_{\R^{2}} f = (\partial_{1} f, \partial_{2} f)^{\top}$, where $\nabla_{\R^{2}}$ is the gradient of $\R^{2}$.

The tangent plane at a point $\y(t, \xi) \in \mathcal{M}_{t}$ is denoted by $T_{\y(t, \xi)}\mathcal{M}_{t}$ 
and the tangent bundle by $T\mathcal{M}_{t} = \left\{ \{\y(t, \xi)\} \times T_{\y(t, \xi)}\mathcal{M}_{t}: \xi \in \Omega \right\}$.
The orthogonal projector onto the tangent plane $T_{x}\mathcal{M}_{t}$ at $x \in \mathcal{M}_{t}$, $t \in I$, is given by
\begin{equation*}
	\mathrm{P}_{\mathcal{M}}(t, x) = \mathrm{Id} - \mathbf{\hat{N}}(t, x) \mathbf{\hat{N}}(t, x)^{\top} \in \R^{3 \times 3}.
\end{equation*}
In particular if $\mathcal{M}_{t} = \mathcal{S}^{2}$, that is $\tilde{\rho}$ in~\eqref{eq:param} is identically one for all $t \in I$, the outward unit normal and the orthogonal projector are given by $\mathbf{\tilde{N}}$ and $\mathrm{P}_{\mathcal{S}^{2}}$, respectively.

In what follows, we define spatial differential operators.
As they are identical to those on static surfaces we consider time $t \in I$ arbitrary but fixed.
Then, the surface gradient of $\hat{f}$, as given in \eqref{eq:f_on_manifold}, is defined by
\begin{equation}
	\nabla_{\mathcal{M}} \hat{f} \coloneqq \mathrm{P}_{\mathcal{M}} \nabla_{\R^3}\hat{\bar{f}} \in \R^3,
\label{eq:surfgrad}
\end{equation}
where $\nabla_{\R^3}$ denotes the usual gradient of the embedding space.
Let us stress that it is independent of the chosen extension, see e.g. \cite[p. 389]{GilTru01}.

We emphasise that, in particular, if $\mathcal{M}_t = \mathcal{S}^{2}$ for all $t \in I$ it follows that
\begin{equation*}
	\nabla_{\R^3} \tilde{\bar{f}} = \mathrm{P}_{\mathcal{S}^{2}} \nabla_{\R^3} \tilde{\bar{f}} + (\mathrm{Id} - \mathrm{P}_{\mathcal{S}^{2}}) \nabla_{\R^3} \tilde{\bar{f}}.
\end{equation*}
The last term of the sum on the right hand side is the normal derivative of $\tilde{\bar{f}}$, which according to the definition of the extension in \eqref{eq:extension} vanishes.
Thus,
\begin{equation}
	\nabla_{\mathcal{S}^2} \tilde{f} = \mathrm{P}_{\mathcal{S}^2} \nabla_{\R^3} \tilde{\bar{f}} = \nabla_{\R^3} \tilde{\bar{f}}.
\label{eq:surfgrad_S}
\end{equation}

For convenience let us observe that, by taking $\partial_{i} f$ in \eqref{eq:f_on_manifold}, we arrive at
\begin{equation}
	\partial_{i} f(t, \xi) = \nabla_{\R^{3}} \tilde{\bar{f}}(t, \x(\xi)) \cdot \partial_{i} \x(\xi) = \nabla_{\mathcal{S}^2} \tilde{f}(t, \x(\xi)) \cdot \partial_{i} \x(\xi)
\label{eq:partialf}
\end{equation}
due to the chain rule and the projection onto the tangent plane $T_{\x(\xi)}\mathcal{S}^2$.

Analogously to the surface gradient we define the spherical Laplace-Beltrami of $\tilde{f}: I \times \mathcal{S}^2 \to \R$ as
\begin{equation}
	\Delta_{\mathcal{S}^2} \tilde{f} = - \Delta_{\R^{3}} \tilde{\bar{f}},
\label{eq:laplacebeltrami}
\end{equation}
where $\Delta_{\R^{3}}$ is the standard Laplacian of $\R^{3}$.

The set
\begin{equation}
	\{ \partial_{1} \y(t, \xi), \partial_{2} \y(t, \xi) \} \subseteq \R^3,
\label{eq:tangentbasis}
\end{equation}
where $\y$ is the parametrisation defined in~\eqref{eq:param}, forms a basis of the tangent space $T_{\y(t, \xi)}\mathcal{M}_{t}$ at $\y(t, \xi)$.
Its elements form the gradient matrix $D \y$, which is derived as follows.

Let $\tilde{\bar{\rho}}$ be the extension of $\tilde{\rho}: I \times \mathcal{S}^{2} \to (0, \infty)$ according to \eqref{eq:extension}.
Then, $\y$ from~\eqref{eq:param} can be rewritten as
\begin{equation*}
	\y(t, \xi) = \tilde{\bar{\rho}}(t, \x(\xi)) \x(\xi).
\end{equation*}
By the chain rule,
\begin{equation*}
	\partial_i \y (t, \xi) = \bigl( \nabla_{\R^3} \tilde{\bar{\rho}}(t, \x(\xi)) \cdot \partial_i \x(\xi) \bigr) \x(\xi) + \tilde{\bar{\rho}}(t, \x(\xi)) \partial_i \x(\xi).
\end{equation*}
Using~\eqref{eq:surfgrad_S} and the fact that $\tilde{\bar{\rho}}$ equals $\tilde{\rho}$ on $\mathcal{S}^2$ gives
\begin{equation*}
	\partial_i \y = \bigl( \nabla_{\mathcal{S}^2} \tilde{\rho} \cdot \partial_i \x \bigr) \x + \tilde{\rho} \partial_i \x,
\end{equation*}
where we have omitted the arguments $(t, \xi)$ and $(\xi)$ for better readability.
Whenever convenient and no confusion will arise we will continue to do so.

By applying~\eqref{eq:partialf} backwards and the fact that $\tilde{\rho}(t, \x(\xi)) = \rho(\xi)$ we have shown
\begin{equation}
\begin{aligned}
	D\y &= \begin{pmatrix}
		\partial_{1} \y & \partial_{2} \y
	\end{pmatrix} \\
	&=
	\begin{pmatrix}
		(\partial_{1} \rho) \x & (\partial_{2} \rho) \x
	\end{pmatrix} + \rho D\x
	\in \R^{3 \times 2},
\end{aligned}
\label{eq:Dy}
\end{equation}
where $D\x = (\partial_{1} \x, \partial_{2} \x)$ is the gradient matrix associated with $\x$.

As a consequence, we can uniquely represent a tangent vector $\mathbf{\hat{u}} \in T_{\y(t, \xi)}\mathcal{M}_{t}$ as $\mathbf{\hat{u}} = \sum_{i=1}^{2} u^{i} \partial_{i} \y$, where $\mathbf{u} = (u^{1}, u^{2})^{\top} \in \R^{2}$ is its coordinate representation, see e.g. \cite[Prop. 3.15]{Lee13}.
We call $u^{i}$ the components of $\mathbf{\hat{u}}$.

In the sequel we will use Einstein summation convention.
We sum over every index letter that appears exactly twice in an expression, once as a sub- and once as a superscript.
For instance, we write $\mathbf{\hat{u}} = u^{i} \partial_{i} \y$ for the sake of brevity.

We underline that the coordinate basis~\eqref{eq:tangentbasis} is not orthogonal in general.
We will, however, require an orthonormal frame $\{ \mathbf{\hat{e}}_{1}(t, \xi), \mathbf{\hat{e}}_{2}(t, \xi) \}$ of the tangent space $T_{\y(t, \xi)}\mathcal{M}_{t}$ from Sec.~\ref{sec:sobolevspaces} onwards.
In the coordinate basis it reads
\begin{equation}
	\mathbf{\hat{e}}_{i} = \alpha_{i}^{j} \partial_{j} \y,
\label{eq:onb}
\end{equation}
where $\alpha_{i}^{j}: I \times \Omega \to \R$, $i, j = \{ 1, 2 \}$, are functions obtained from the Gram-Schmidt process.

Combining~\eqref{eq:f_on_manifold} and~\eqref{eq:partialf} with the expressions derived for $D\x$ and $D\y$ we can conveniently state that
\begin{equation}
	\nabla_{\R^{2}} f = D\x^{\top} \nabla_{\mathcal{S}^{2}} \tilde{f} \text{ and } \nabla_{\R^{2}} f = D\y^{\top} \nabla_{\mathcal{M}} \hat{f}.
\label{eq:gradf}
\end{equation}

Let us derive the following useful generalisation of~\eqref{eq:partialf}.
For a tangent vector $\mathbf{\tilde{v}} = v^{i} \partial_{i} \x \in T_{x}\mathcal{S}^{2}$, $x \in \mathcal{S}^{2}$, the directional derivative of $\tilde{f}$ along $\mathbf{\tilde{v}}$ at $x$ is
\begin{equation}
	\nabla_{\mathcal{S}^{2}} \tilde{f} \cdot \mathbf{\tilde{v}} = \nabla_{\mathcal{S}^{2}} \tilde{f} \cdot v^{i} \partial_{i} \x = (D\x^{\top} \nabla_{\mathcal{S}^{2}} \tilde{f}) \cdot \mathbf{v} = \nabla_{\R^{2}} f \cdot \mathbf{v} = v^{i} \partial_{i} f,
\label{eq:directionalderiv_S}
\end{equation}
where the third equality follows from the first equation in~\eqref{eq:gradf}.
Analogously, for $\mathbf{\hat{v}} = v^{i} \partial_{i} \y \in T_{x}\mathcal{M}_{t}$, with $x \in \mathcal{M}_{t}$ and $t \in I$, one can derive
\begin{equation}
	\nabla_{\mathcal{M}} \hat{f} \cdot \mathbf{\hat{v}} = v^{i} \partial_{i} f.
\label{eq:directionalderiv_N}
\end{equation}
As soon as we have established the relation between $\mathbf{\hat{v}}$ and $\mathbf{\tilde{v}}$ it will conveniently allow us to switch between~\eqref{eq:directionalderiv_S} and~\eqref{eq:directionalderiv_N}.

Moreover, the coordinate representation of the surface gradient~\eqref{eq:surfgrad_S} is derived as follows.
Let us start out with the first equation in~\eqref{eq:gradf}.
By writing $\nabla_{\mathcal{S}^2} \tilde{f}$ in the coordinate basis, that is $\nabla_{\mathcal{S}^2} \tilde{f} = D\x \mathbf{u}$ for some $\mathbf{u}$, we obtain from~\eqref{eq:gradf}
\begin{equation*}
	\nabla_{\R^{2}} f = D\x^{\top} D\x \mathbf{u}.
\end{equation*}
Multiplying with $(D\x^{\top} D\x)^{-1}$ from the left yields
\begin{equation*}
	(D\x^{\top} D\x )^{-1} \nabla_{\R^{2}} f = \mathbf{u}.
\end{equation*}
Thus,
\begin{equation}
	\nabla_{\mathcal{S}^2} \tilde{f} = D\x \mathbf{u} = D\x (D\x^{\top} D\x )^{-1} \nabla_{\R^{2}} f.
\label{eq:surfgrad_S_coord}
\end{equation}

Furthermore, let $t \in I$ be fixed and let $\hat{f}(t, \cdot): \mathcal{M}_{t} \to \R$.
The surface integral of $\hat{f}$ is
\begin{equation}
	\int_{\mathcal{M}_{t}} \hat{f}  \; d\mathcal{M}_{t} = \int_{\Omega} f J\y \; d\xi,
\label{eq:surfintegral}
\end{equation}
where $J\y$ is the Jacobian of $\y$.
According to Theorem~3 in \cite[p. 88]{EvaGar92}, it is given by
\begin{equation*}
	(J\y)^{2} = \det(D\y^{\top} D\y)
\end{equation*}
and by using~\eqref{eq:Dy} yields
\begin{equation}
\begin{aligned}
	(J\y)^{2} & = \rho^{2} \left.\bigl( (\partial_{1} \rho)^{2} \partial_{2} \x \cdot \partial_{2} \x + (\partial_{2} \rho)^{2} \partial_{1} \x \cdot \partial_{1} \x \right. \\
	& \left. + \rho^{2} (\partial_{1} \x \cdot \partial_{1} \x) (\partial_{2} \x \cdot \partial_{2} \x) - 2 \partial_{1} \rho \partial_{2} \rho (\partial_{1} \x \cdot \partial_{2} \x) - \rho^{2} (\partial_{1} \x \cdot \partial_{2} \x)^{2} \right.\bigr).
\end{aligned}
\label{eq:jacobian}
\end{equation}
Note that $\x \cdot \x = 1$ and thus, terms of the form $\partial_{i} \x \cdot \x$ vanish.
By the differentiability of $\x$, ones has
\begin{equation*}
	\partial_{i} (\x \cdot \x) = 0.
\end{equation*}
Therefore, $\partial_{i} \x \cdot \x = 0$, meaning that tangential and normal vectors are orthogonal.
We emphasise that $D\y^{\top} D\y$ is commonly referred to as \emph{Riemannian metric}.
It is positive definite and thus, $(J\y(t, \xi))^{2} > 0$ for all $(t, \xi) \in I \times \Omega$.

The parametrisations $\x$ and $\y$ defined in~\eqref{eq:sphereparam} and~\eqref{eq:param}, respectively, suggest the straightforward construction of a smooth map $\tilde{\phi}(t, \cdot): \mathcal{S}^{2} \to \mathcal{M}_{t}$.
It is given by the composition $(\y \circ \x^{-1})(t, \cdot)$, that is
\begin{equation*}
	\tilde{\phi}(t, x): x \mapsto \tilde{\rho}(t, x)x.
\end{equation*}
The differential $D\tilde{\phi}(t, x): T_{x}\mathcal{S}^2 \to T_{\tilde{\phi}(t, x)}\mathcal{M}_{t}$ of $\tilde{\phi}$ is a linear map and is given by
\begin{equation}
	D\tilde{\phi}(t, x) = \tilde{\rho}(t, x) \mathrm{Id} + x \nabla_{\mathcal{S}^2} \tilde{\rho}(t, x)^{\top} \in \R^{3 \times 3}.
\label{eq:differential}
\end{equation}
It follows from a direct calculation akin to the derivation of $D\y$ in~\eqref{eq:Dy}.

Let us exhibit the action of $D\tilde{\phi}(t, x)$, for $x = \x(\xi)$ and $t \in I$, onto a tangent vector $\mathbf{\tilde{v}} = v^{i} \partial_{i} \x \in T_{x}\mathcal{S}^2$.
We have
\begin{equation}
\begin{aligned}
	D\tilde{\phi}(t, x)(\mathbf{\tilde{v}}) & = \tilde{\rho}(t, x) \mathbf{\tilde{v}} + x (\nabla_{\mathcal{S}^2} \tilde{\rho}(t, x) \cdot \mathbf{\tilde{v}}) \\
	& \overset{\mathclap{\eqref{eq:directionalderiv_S}}}{=} \tilde{\rho}(t, x) \mathbf{\tilde{v}} + x v^{i} \partial_{i} \rho(\xi) \\
	& = \tilde{\rho}(t, x) v^{i} \partial_{i} \x + x v^{i} \partial_{i} \rho(\xi) \\
	& = v^{i} \bigl( \tilde{\rho}(t, x) \partial_{i} \x + x \partial_{i} \rho(\xi) \bigr) \\
	& \overset{\mathclap{\eqref{eq:Dy}}}{=} v^{i} \partial_{i} \y(\xi).
\end{aligned}
\label{eq:Dphiv}
\end{equation}
In other words, the components $(v^{1}, v^{2})^{\top}$ are preserved whenever a tangent vector is mapped from $\mathcal{S}^2$ to $\mathcal{M}_{t}$ via the differential~\eqref{eq:differential}.

As a matter of fact, given a tangent vector field $\mathbf{\tilde{v}} = v^{i} \partial_{i} \x$ on $\mathcal{S}^2$, the differential $D\tilde{\phi}$ gives rise to a unique tangent vector field $\mathbf{\hat{v}} = v^{i} \partial_{i} \y$ on $\mathcal{M}_{t}$, see~\cite[Chapter 8]{Lee13}.
Whenever we use $\mathbf{\tilde{v}}$ and $\mathbf{\hat{v}}$ in the sequel we refer to their unique identification via the differential~\eqref{eq:differential} and call $\mathbf{\hat{v}}$ the pushforward of $\mathbf{\tilde{v}}$.
At this point, the reader might find it helpful to have a look at Fig.~\ref{fig:cd}.

With the above definitions at hand we are able to relate the surface integral~\eqref{eq:surfintegral} to an integral on $\mathcal{S}^{2}$ via a change of variables.
The key is to compute a meaningful surface element as $\abs{\det(D\tilde{\phi})}$ is the magnitude of the change of the volume element.
The following lemma provides the required form.
\begin{lemma}
\label{lem:surfintegral}
Let $\x: [0, \pi] \times [0, 2\pi) \to \R^{3}$ be the standard parametrisation of $\mathcal{S}^2$,
\begin{equation*}
	(\xi^{1}, \xi^{2})^{\top} \mapsto (\sin{\xi^{1}} \cos{\xi^{2}}, \sin{\xi^{1}} \sin{\xi^{2}}, \cos{\xi^{1}} )^{\top},
\end{equation*}
and let $\hat{f}: \mathcal{M} \to \R$ and $\tilde{\rho}: I \times \mathcal{S}^2 \to (0, \infty)$ be as above.
Then, for $t \in I$,
\begin{equation*}
	\int_{\mathcal{M}_{t}} \hat{f} \, d\mathcal{M}_{t} = \int_{\mathcal{S}^2} \tilde{f} \tilde{\rho} \sqrt{\norm{\nabla_{\mathcal{S}^2} \tilde{\rho}}^{2} + \tilde{\rho}^{2}} \, d\mathcal{S}^2.
\end{equation*}
\end{lemma}
\begin{proof}
Let us denote by $\mathbf{\tilde{e}}_{1}(\xi)$ and $\mathbf{\tilde{e}}_{2}(\xi)$ the orthogonal unit vectors on $\mathcal{S}^{2}$ in direction of $\xi^{1}$ and $\xi^{2}$, respectively, which are obtained by normalising the coordinate basis $\{ \partial_{1} \x(\xi), \partial_{2} \x(\xi) \}$.
That is,
\begin{equation}
	\mathbf{\tilde{e}}_{1}(\xi) = \partial_{1} \x(\xi) \text{ and } \mathbf{\tilde{e}}_{2}(\xi) = \frac{\partial_{2} \x(\xi)}{\norm{\partial_{2} \x(\xi)}}.
\label{eq:onb_S}
\end{equation}
Moreover, a straightforward calculation gives
\begin{equation*}
	D\x^{\top} D\x = \begin{pmatrix}
		1 & 0 \\
		0 & \sin^{2}{\xi^{1}}
	\end{pmatrix}
\end{equation*}
and thus, the surface gradient of $\tilde{\rho}$ in spherical coordinates~\eqref{eq:surfgrad_S_coord} is given by
\begin{align*}
	\nabla_{\mathcal{S}^2} \tilde{\rho}(t, \x(\xi)) & = \partial_{1} \rho(\xi) \, \partial_{1} \x(\xi) + \frac{1}{\sin^{2}{\xi^{1}}} \partial_{2} \rho(\xi) \, \partial_{2} \x(\xi) \\
	& \overset{\mathclap{\eqref{eq:onb_S}}}{=} \partial_{1} \rho(\xi) \, \mathbf{\tilde{e}}_{1}(\xi) + \frac{1}{\sin{\xi^{1}}} \partial_{2} \rho(\xi) \, \mathbf{\tilde{e}}_{2}(\xi),
\end{align*}
where we have replaced the coordinate basis with the orthonormal basis.

Using $D\x^{\top} D\x$ in~\eqref{eq:jacobian}, the Jacobian $J\y$ can be written as
\begin{align*}
	(J\y)^{2} & = \rho^{2} \bigl( (\partial_{1} \rho)^{2} \sin^{2}{\xi^{1}} + (\partial_{2} \rho)^{2} + \rho^{2} \sin^{2}{\xi^{1}} \bigr) \\
	& = \rho^{2} \bigl( (\partial_{1} \rho)^{2} + \frac{1}{\sin^{2}{\xi^{1}}}(\partial_{2} \rho)^{2} + \rho^{2} \bigr) \sin^{2}{\xi^{1}} \\
	& = \rho^{2} \bigl( \norm{\nabla_{\mathcal{S}^2} \tilde{\rho}}^{2} + \rho^{2} \bigr) \sin^{2}{\xi^{1}}.
\end{align*}
Here, we have omitted the argument $(t, \x(\xi))$ of $\nabla_{\mathcal{S}^2} \tilde{\rho}$.
Then, the integral turns out to be
\begin{align*}
	\int_{\mathcal{M}_{t}} \hat{f} \, d\mathcal{M}_{t} & = \int_{0}^{2\pi} \int_{0}^{\pi} f J\y \; d\xi \\
	& = \int_{0}^{2\pi} \int_{0}^{\pi} f \rho \sqrt{\norm{\nabla_{\mathcal{S}^2} \tilde{\rho}}^{2} + \rho^{2}} \sin{\xi^{1}} \; d\xi \\
	& = \int_{\mathcal{S}^2} \tilde{f} \tilde{\rho} \sqrt{\norm{\nabla_{\mathcal{S}^2} \tilde{\rho}}^{2} + \tilde{\rho}^{2}} \, d\mathcal{S}^2,
\end{align*}
where the last equation follows from \eqref{eq:surfintegral} if $\mathcal{M}_{t} = \mathcal{S}^{2}$, the fact that $\sin{\xi^{1}} \ge 0$, and
\begin{equation*}
	J\x = \sqrt{\det(D\x^{\top} D\x)} = \sin{\xi^{1}}.
\end{equation*}
\end{proof}

The concepts introduced above, and further properties thereof, may be found in any standard differential geometry book.
For instance, in~\cite{Car76, Car92, Lee97, Lee13}.

\subsection{Vectorial Sobolev Spaces on Manifolds} \label{sec:sobolevspaces}

We briefly introduce the appropriate function spaces required for the variational optical flow formulation on Riemannian manifolds.
Again, let us consider time $t \in I$ arbitrary but fixed.

For a tangent vector field $\mathbf{\hat{v}}$ on $\mathcal{M}_{t}$ we denote by $\nabla_{\mathbf{\hat{u}}} \mathbf{\hat{v}}(x)$ the covariant derivative of $\mathbf{\hat{v}}$ at $x \in \mathcal{M}_{t}$ along the direction of a tangent vector $\mathbf{\hat{u}} \in T_{x}\mathcal{M}_{t}$.
We define it as the tangential part of the usual directional derivative of the extension $\mathbf{\hat{\bar{v}}}$ along $\mathbf{\hat{u}}$ in the embedding space, that is,
\begin{equation}
	\nabla_{\mathbf{\hat{u}}} \mathbf{\hat{v}} \coloneqq \mathrm{P}_{\mathcal{M}} \nabla_{\R^3} \mathbf{\hat{\bar{v}}} (\mathbf{\hat{u}}).
\label{eq:covderiv}
\end{equation}

It is a linear operator $\nabla \mathbf{\hat{v}}(x): T_{x}\mathcal{M}_{t} \to T_{x}\mathcal{M}_{t}$ and its Hilbert-Schmidt norm is given by
\begin{equation}
	\norm{\nabla \mathbf{\hat{v}}(x)}_{2}^{2} = \sum_{i=1}^{2} \norm{\nabla_{\mathbf{\hat{e}}_{i}} \mathbf{\hat{v}}(x)}^{2},
\label{eq:hsnorm}
\end{equation}
where $\{ \mathbf{\hat{e}}_{1}, \mathbf{\hat{e}}_{2} \}$ denotes the orthonormal basis of the tangent space $T_{x}\mathcal{M}_{t}$, cf.~\eqref{eq:onb}.
We stress that~\eqref{eq:hsnorm} is invariant with respect to the chosen parametrisation.

For each $t \in I$, we define the Sobolev space $H^{1}(\mathcal{M}_{t}, T\mathcal{M}_{t})$ as the completion of the space of vector fields $C^{\infty}(\mathcal{M}_{t}, T\mathcal{M}_{t})$ with respect to the norm
\begin{equation}
	\norm{\mathbf{\hat{v}}}_{H^{1}(\mathcal{M}_{t}, T\mathcal{M}_{t})}^{2} \coloneqq \int_{\mathcal{M}_{t}} \norm{\nabla \mathbf{\hat{v}}}_{2}^{2} \; d\mathcal{M}_{t},
\label{eq:sobolevnorm}
\end{equation}
where the surface integral is defined in \eqref{eq:surfintegral}.
Let us emphasise that~\eqref{eq:sobolevnorm} is indeed a norm whenever $\mathcal{M}_{t}$ is diffeomorphic to the 2-sphere.
The reason is that, by virtue of the Hairy Ball Theorem, there is no covariantly constant tangent vector field but $\mathbf{\hat{v}} = 0$, see e.g.~\cite[p. 125]{Hir94}.

Alternatively, one can define Sobolev spaces of vector fields such that each component of a vector field originates from a scalar Sobolev space.
See, for instance, Lef\`{e}vre and Baillet~\cite{LefBai08}.
On the 2-sphere, however, they are typically introduced by means of the spherical Laplace-Beltrami operator, see e.g.~\cite[Chapter~6.2]{Mic13} and Sec.~\ref{sec:spharm} for the scalar counterpart.
For a thorough treatment of Sobolev spaces on Riemannian manifolds we refer to the books~\cite{Heb96, Tri92}.

\subsection{Spherical Harmonics} \label{sec:spharm}

We denote by $\mathrm{Harm}_{n}$ the space of homogeneous harmonic polynomials of degree $n \in \mathbb{N}_{0}$ with their domain restricted to $\mathcal{S}^2$.
Its dimension is
\begin{equation*}
	\dim(\mathrm{Harm}_{n}) = 2n + 1.
\end{equation*}
An element $\tilde{Y}_{n} \in \mathrm{Harm}_{n}$, $n \in \mathbb{N}_{0}$, is called a \emph{(scalar) spherical harmonic}.
It is an infinitely often differentiable eigenfunction of the Laplace-Beltrami operator~$\Delta_{\mathcal{S}^2}$, defined in~\eqref{eq:laplacebeltrami}, with corresponding eigenvalue
\begin{equation*}
	\lambda_{n} = n(n + 1).
\end{equation*}
We refer to Theorem~5.6 and Lemma~5.8 in~\cite[Sec.~5.1]{Mic13} for detailed proofs of the previous statements.

The set
\begin{equation}
	\Big\lbrace \tilde{Y}_{nj}: n \in \mathbb{N}_{0}, j = 1, \dots, 2n + 1 \Big\rbrace
\label{eq:spharm}
\end{equation}
is a complete orthonormal system of $L^{2}(\mathcal{S}^2)$ with respect to the inner product $\langle \cdot, \cdot \rangle_{L^{2}(\mathcal{S}^2)}$ on $\mathcal{S}^2$.
In further consequence, for a function $\tilde{f} \in L^{2}(\mathcal{S}^2)$, we have the Fourier series representation
\begin{equation*}
	\tilde{f} = \sum_{n = 0}^{\infty} \sum_{j=1}^{2n + 1} \langle \tilde{f}, \tilde{Y}_{nj} \rangle_{L^{2}(\mathcal{S}^2)} \tilde{Y}_{nj},
\end{equation*}
Again, we refer to~\cite[Sec.~5.1]{Mic13} for the proofs, in particular to Theorem~5.25.
In the present article we employ \emph{fully normalised spherical harmonics}.
For the explicit construction see~\cite[Sec.~5.2]{Mic13}.

Moreover, the norm of $L^{2}(\mathcal{S}^2)$ is readily stated in terms of the coefficients in the above expansion via Parseval's identity
\begin{equation*}
	\norm{\tilde{f}}_{L^{2}(\mathcal{S}^2)}^{2} = \sum_{n, j} \langle \tilde{f}, \tilde{Y}_{nj} \rangle_{L^{2}(\mathcal{S}^2)}^{2}.
\end{equation*}

For an arbitrary real number $s \in \R$, we define the Sobolev space $H^{s}(\mathcal{S}^2)$ as the completion of all $C^{\infty}(\mathcal{S}^2)$ functions with respect to the norm
\begin{equation*}
	\norm{\tilde{f}}_{H^{s}(\mathcal{S}^2)}^{2} \coloneqq \norm{(\Delta_{\mathcal{S}^2} + 1)^{s/2} \tilde{f}}_{L^{2}(\mathcal{S}^2)}^{2} = \sum_{n, j} (\lambda_{n} + 1)^{s}  \langle \tilde{f}, \tilde{Y}_{nj} \rangle_{L^{2}(\mathcal{S}^2)}^{2}.
\end{equation*}
We stress that, by~\eqref{eq:laplacebeltrami}, $\Delta_{\mathcal{S}^{2}} \tilde{f} = - \Delta_{\R^{3}} \tilde{\bar{f}}$ and we have $\lambda_{n} \ge 0$ for all $n \in \mathbb{N}_{0}$ yielding a sound definition.
Accordingly, for $s \in \R$, we define the $H^{s}$ seminorm of order $s$ by
\begin{equation}
	\abs{\tilde{f}}_{H^{s}(\mathcal{S}^2)}^{2} \coloneqq \norm{\Delta_{\mathcal{S}^2}^{s/2} \tilde{f}}_{L^{2}(\mathcal{S}^2)}^{2} = \sum_{n, j} \lambda_{n}^{s}  \langle \tilde{f}, \tilde{Y}_{nj} \rangle_{L^{2}(\mathcal{S}^2)}^{2}.
\label{eq:sobolevseminorm}
\end{equation}

Now that the space $L^{2}(\mathcal{S}^2)$ is endowed with a basis, we can proceed to define an orthonormal system for square integrable tangent vector fields on the sphere.
This will immediately allow us to treat vector-valued problems consistently.

Let $\tilde{Y}_{n} \in \mathrm{Harm}_{n}$ be a scalar spherical harmonic of degree $n \in \mathbb{N}_{0}$.
Any vector field $\mathbf{\tilde{y}}: \mathcal{S}^2 \to \R^{3}$ that can be written in the form $\mathbf{\tilde{y}} = \mathbf{\tilde{y}}_{n}^{(i)}$, where
\begin{align*}
	\mathbf{\tilde{y}}_{n}^{(1)} &\coloneqq \tilde{Y}_{n} \mathbf{\tilde{N}}, \\
	\mathbf{\tilde{y}}_{n}^{(2)} &\coloneqq \nabla_{\mathcal{S}^2} \tilde{Y}_{n}, \\
	\mathbf{\tilde{y}}_{n}^{(3)} &\coloneqq \nabla_{\mathcal{S}^2} \tilde{Y}_{n} \times \mathbf{\tilde{N}},
\end{align*}
is called a \emph{vector spherical harmonic} of degree $n$ and type $i$, cf.~\cite[Definition~5.2]{FreSchr09}.
Recall that $\mathbf{\tilde{N}}$ is the outward unit normal to $\mathcal{S}^2$.

By definition, $\mathbf{\tilde{y}}_{n}^{(1)}$ is a normal field whereas $\mathbf{\tilde{y}}_{n}^{(2)}$ and $\mathbf{\tilde{y}}_{n}^{(3)}$ are tangent vector fields.
Consequently, the latter are called \emph{tangential vector spherical harmonics}.
Note that, by means of the Hairy-Ball Theorem, no tangential vector spherical harmonics of degree zero exist.

In further consequence, let us denote by $L^{2}(\mathcal{S}^2, T\mathcal{S}^2)$ the space of square integrable tangent vector fields on $\mathcal{S}^2$ equipped with the inner product
\begin{equation*}
	\langle \mathbf{\tilde{u}}, \mathbf{\tilde{v}} \rangle_{L^{2}(\mathcal{S}^2, T\mathcal{S}^2)} = \int_{\mathcal{S}^2} \mathbf{\tilde{u}} \cdot \mathbf{\tilde{v}} \; d\mathcal{S}^2.
\end{equation*}
Here, $d\mathcal{S}^2$ denotes the usual spherical surface measure, see also Lemma~\ref{lem:surfintegral}.

Since~\eqref{eq:spharm} is an orthonormal set for $L^{2}(\mathcal{S}^2)$, the set
\begin{equation}
	\Big\lbrace \mathbf{\tilde{y}}_{nj}^{(i)}: n \in \mathbb{N}, j = 1, \dots, 2n + 1, i = 2, 3 \Big\rbrace,
\label{eq:vspharm}
\end{equation}
is an orthonormal system for $L^{2}(\mathcal{S}^2, T\mathcal{S}^2)$, where we have defined
\begin{equation}
\begin{aligned}
	\mathbf{\tilde{y}}_{nj}^{(2)} & = \lambda_{n}^{-1/2} \nabla_{\mathcal{S}^2} \tilde{Y}_{nj}, \\
	\mathbf{\tilde{y}}_{nj}^{(3)} & = \lambda_{n}^{-1/2} \nabla_{\mathcal{S}^2} \tilde{Y}_{nj} \times \mathbf{\tilde{N}},
\end{aligned}
\label{eq:fullynormalisedvspharm}
\end{equation}
for orthonormalisation purpose, see~\cite[Sec.~5.2]{FreSchr09}.
Thus, every vector field $\mathbf{\tilde{v}} \in L^{2}(\mathcal{S}^2, T\mathcal{S}^2)$ can be written uniquely as
\begin{equation*}
	\mathbf{\tilde{v}} = \sum_{i=2}^{3} \sum_{n=1}^{\infty} \sum_{j=1}^{2n+1} \langle \mathbf{\tilde{v}}, \mathbf{\tilde{y}}_{nj}^{(i)} \rangle_{L^{2}(\mathcal{S}^2, T\mathcal{S}^2)} \mathbf{\tilde{y}}_{nj}^{(i)}.
\end{equation*}

We refer to the books~\cite{FreSchr09, Mic13} for further details on the subject.
Table~\ref{tab:notation} contains a summary of notation used in the sequel.

\begin{table}[t]
\begin{tabular}{ll}
	$\Omega$ & coordinate domain \\
	$I$ & time interval \\
	$\mathcal{S}^2$ & 2-sphere \\
	$\mathcal{M}$ & family of sphere-like surfaces $\mathcal{M}_{t}$ \\
	$T_{x}\mathcal{S}^2$ & tangent plane at $x \in \mathcal{S}^2$ \\
	$T_{y}\mathcal{M}_{t}$ & tangent plane at $y \in \mathcal{M}_{t}$ \\
	$\mathbf{\tilde{N}}, \mathbf{\hat{N}}$ & outward unit normals to $\mathcal{S}^2$ and $\mathcal{M}$ \\
	$\x$, $\y$ & parametrisations of $\mathcal{S}^2$ and $\mathcal{M}$ \\
	$D\x$, $D\y$ & gradient matrix of $\x$ and $\y$ \\
	$\{ \partial_{1} \x, \partial_{2} \x \}$ & basis for $T\mathcal{S}^2$ \\
	$\{ \partial_{1} \y, \partial_{2} \y \}$ & basis for $T\mathcal{M}$ \\
	$\{ \mathbf{\hat{e}}_{1}, \mathbf{\hat{e}}_{2} \}$ & orthonormal basis for $T\mathcal{M}_{t}$ \\
	$\mathbf{\hat{V}}$ & surface velocity of $\mathcal{M}$ \\
	$\tilde{\phi}, D\tilde{\phi}$ & smooth map from $\mathcal{S}^{2}$ to $\mathcal{M}$ and its differential \\
	$\tilde{f}$, $\hat{f}$, $f$ & scalar function on $\mathcal{S}^2$, $\mathcal{M}$, and their coordinate version \\
	$\nabla_{\mathcal{S}^2} \tilde{f}$, $\nabla_{\mathcal{M}} \hat{f}$ & surface gradient on $\mathcal{S}^2$ and $\mathcal{M}_{t}$ \\
	$\mathbf{\tilde{v}}$, $\mathbf{\hat{v}}$, $\mathbf{v}$ & tangent vector fields on $\mathcal{S}^2$, $\mathcal{M}$, and their coordinate version \\
	$\nabla_{\mathbf{\hat{u}}} \mathbf{\hat{v}}$ & covariant derivative of $\mathbf{\hat{v}}$ along direction $\mathbf{\hat{u}}$ on $\mathcal{M}_{t}$ \\
	$\tilde{Y}_{nj}$ & scalar spherical harmonic of degree $n$ and order $j$ \\
	$\mathbf{\tilde{y}}_{nj}^{(i)}$ & vector spherical harmonic of degree $n$, order $j$, and type $i$ \\
	$\mathbf{\hat{y}}_{nj}^{(i)}$ & pushforward of $\mathbf{\tilde{y}}_{nj}^{(i)}$ via the differential $D\tilde{\phi}$
\end{tabular}
\caption{Summary of notation used throughout the paper.}
\label{tab:notation}
\end{table}
\section{Optical Flow on Evolving Surfaces} \label{sec:model}

\subsection{Generalised Optical Flow}

Optical flow models are typically based on the assumption of \emph{constant brightness}.
Given a sequence of (planar) images
\begin{equation*}
	f: I \times \Omega \subset \R \times \R^{2} \to \R
\end{equation*}
such that $f \in C^{1}(I \times \Omega)$, it assumes that the intensity $f(t, \gamma(t, \xi))$ stays constant over time when moving along a trajectory $\gamma(\cdot, \xi): I \to \Omega$ starting at $\xi \in \Omega$.
In other words, in the planar setting, we have
\begin{equation*}
	\frac{d}{dt} f(t, \gamma(t, \xi)) = \partial_{t} f + \nabla_{\R^{2}} f \cdot \partial_{t} \gamma = 0,
\end{equation*}
which is termed \emph{optical flow equation} and must hold for all $\xi \in \Omega$ and all $t \in I$.
For the sake of consistency, we denote by $\partial_{t}$ the partial and by $d/dt$ the total derivative with respect to time.

It is possible to generalise the idea to a non-Euclidean setting where the image lives on a, potentially moving, manifold.
To this end, let us be given an evolving surface
\begin{equation*}
	\mathcal{M} \coloneqq \bigcup_{t \in I} \; \bigl( \{ t \} \times \mathcal{M}_{t} \bigr) \subset \R^{4}
\end{equation*}
specified by a parametrisation $\y: I \times \Omega \to \R^{3}$ as in \eqref{eq:param} together with a function $\hat{f}$, its domain being $\mathcal{M}$.
For a time $t \in I$,
\begin{equation*}
	\hat{f}(t, \cdot): \mathcal{M}_{t} \to \R
\end{equation*}
is then an image on the surface.
Adapting the above idea of constant brightness to the new setting requires that, along a smooth trajectory $\gamma(\cdot, x): t \mapsto \gamma(t, x) \in \mathcal{M}_{t}$ that starts at $x \in \mathcal{M}_{0}$ and always stays on the surface, we must have
\begin{equation}
	\hat{f}(t, \gamma(t, x)) = \hat{f}(0, x).
\label{eq:bca_gamma}
\end{equation}
However, in order to proceed as above one needs to define a meaningful derivative with respect to time.

One possibility, which is pursued in~\cite{KirLanSch13, KirLanSch15}, is to consider derivatives along trajectories following the moving surface.
Let $\y$ be as above and let $\partial_{t} \y = \mathbf{\hat{V}}$ be the surface velocity, its domain being $\bigcup_{t \in I} (\{ t \} \times \mathcal{M}_{t}) \subset \R^{4}$.
We emphasise that $\mathbf{\hat{V}}$ is in general not tangent to $\mathcal{M}_{t}$, $t \in I$, and hence in our notation denoted by a boldface capital letter.
Then,
\begin{equation}
	d_{t}^{\mathbf{\hat{V}}} \hat{f}(t_{0}, x_{0}) \coloneqq \frac{d}{dt} \hat{f}(t, \y(t, \xi)) \bigg\vert_{t = t_{0}}
\label{eq:timederiv}
\end{equation}
is the time derivative of $\hat{f}$ at $x_{0} = \y(t_{0}, \xi)$ along the parametrisation $\y(\cdot, \xi)$.
As a consequence, one can deduce that
\begin{equation*}
	d_{t}^{\mathbf{\hat{V}}} \hat{f} = d_{t}^{\mathbf{\hat{N}}} \hat{f} + \nabla_{\mathcal{M}} \hat{f} \cdot \mathbf{\hat{V}}
\end{equation*}
holds, where $d_{t}^{\mathbf{\hat{N}}} \hat{f}(t_{0}, x_{0})$ is the time derivative of $\hat{f}$ in normal direction.
It is defined analogously to~\eqref{eq:timederiv} albeit following a trajectory $\psi_{\mathbf{\hat{N}}}$ through $x_{0} \in \mathcal{M}_{t_{0}}$ for which $\partial_{t} \psi_{\mathbf{\hat{N}}}(t_{0}, x_{0})$ is orthogonal to $T_{x_{0}}\mathcal{M}_{t_{0}}$.

From that one can immediately formulate the above idea of constant brightness~\eqref{eq:bca_gamma} along $\gamma$.
To this end, we define by $\mathbf{\hat{M}} \coloneqq \partial_{t} \gamma$ the velocity of a point moving along the trajectory $\gamma$ and demand that
\begin{equation}
	d_{t}^{\mathbf{\hat{M}}} \hat{f} = d_{t}^{\mathbf{\hat{N}}} \hat{f} + \nabla_{\mathcal{M}} \hat{f} \cdot \mathbf{\hat{M}} = 0
\label{eq:bca}
\end{equation}
must hold.
Equation~\eqref{eq:bca} is a \emph{generalised optical flow equation}.
In Fig.~\ref{fig:sketch} we sketch the various trajectories through the evolving surface and their corresponding velocities.

Since we are, however, interested in a coordinate representation of $\gamma$, we define a family of trajectories $\beta: I \times \Omega \to \Omega$ such that
\begin{equation*}
	\gamma(t, \y(0, \xi)) = \y(t, \beta(t, \xi))
\end{equation*}
holds for all $t \in I$ and all $\xi \in \Omega$.
In other words, we want the composition of $\beta$ with $\y$, and $\gamma$ to coincide.
By taking the total derivative $d/dt$ on both sides of the above equation we get
\begin{equation*}
	\partial_{t} \gamma = \partial_{t} \y + \partial_{t} \beta^{i} \partial_{i} \y.
\end{equation*}
Let us denote $\mathbf{\hat{v}} \coloneqq \partial_{t} \beta^{i} \partial_{i} \y$ and recall that $\partial_{t} \y = \mathbf{\hat{V}}$ is the surface velocity.
The above relation states that the total velocity $\mathbf{\hat{M}} = \partial_{t} \gamma$ along a level line of constant intensity is
\begin{equation}
	\mathbf{\hat{M}} = \mathbf{\hat{V}} + \mathbf{\hat{v}}
\label{eq:totalmotion}
\end{equation}
and $\mathbf{\hat{v}}$ is a tangential velocity relative to the prescribed surface velocity $\mathbf{\hat{V}}$.

\begin{figure}[t]
	\begin{center}
	\begin{tikzpicture}
		\draw [thick, gray] (-3,0) to [out=30,in=150] (3,0);
		\node [right] at (3,0) {$\mathcal{M}_{t_0}$};		
		\draw [thick, gray] (-3,1) to [out=40,in=150] (4,1.8);
		\node [below] at (4,1.8) {$\mathcal{M}_{t_0+\Delta t}$};
		\draw [-stealth', gray, thick] (0,.9) -- (.91*1.1,.91*2.6);
		\node [left, gray] at (.91*1.1,.91*2.6) {$\mathbf{\hat{V}}$};
		\draw [-stealth', gray, thick] (0,.9) -- (0,2.4);
		\draw [-stealth', gray, thick] (0,.9) -- (.94*2.5,.94*2.4);
		\node [below, gray] at (.98*2.5,.94*2.4) {$\mathbf{\hat{M}}$};
		\draw [-stealth', gray, thick] (0,.9) -- (.94*2.5-.91*1.1, 0.9);
		\node [right, gray] at (0.94*2.5-0.91*1.1, 0.9) {$\mathbf{\hat{v}}$};
		\draw [-stealth', thick, dotted] (0,0) to [out=90,in=225] (0,.9) to [out=45,in=270] (1.3,1.7) to [out=90,in=225] (2.5,2.4) to [out=45,in=225] (1.25*2.5,1.25*2.4);
		\node [right] at (1.3*2.5,1.3*2.4) {$\gamma(\cdot, x_{0})$};
		\draw [-stealth', thick] (0.5,0) to [out=135,in=270] (0,.9) to [out=90,in=300] (-.5,1.25*2.4);
		\node [above] at (-.5,1.25*2.4) {$\psi_{\mathbf{\hat{N}}}$};
		\draw [-stealth', thick, dashed] (-0.5,0) to [out=45,in=225] (0,.9) to [out=45,in=270] (1.5,3.1);
		\node [above] at (1.5,3.1) {$\y(\cdot, \xi)$};
		\node at (0.4, 0.6) {$x_{0}$};
	\end{tikzpicture}
	\end{center}
	\caption{Illustration of trajectories through the evolving surface. Their corresponding velocities are shown in grey.}
	\label{fig:sketch}
\end{figure}
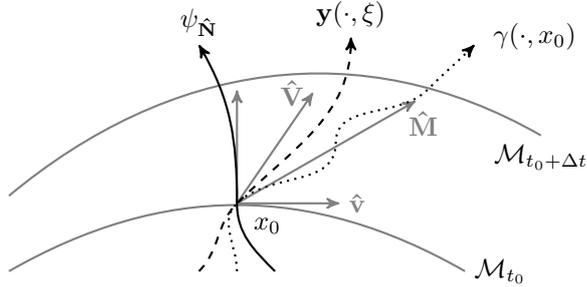

Solving the generalised optical flow equation~\eqref{eq:bca}, however, is inconvenient as $\psi_{\mathbf{\hat{N}}}$ and, in further consequence, $d_{t}^{\mathbf{\hat{N}}}$ is unknown or hard to estimate.
Nevertheless, one can relate~\eqref{eq:bca} and~\eqref{eq:totalmotion}, as shown in~\cite[Lemma~2]{KirLanSch15}, and arrive at the \emph{parametrised optical flow equation}
\begin{equation}
	d_{t}^{\mathbf{\hat{V}}} \hat{f} + \nabla_{\mathcal{M}} \hat{f} \cdot \mathbf{\hat{v}} = 0.
\label{eq:ofc}
\end{equation}
Solving for the optical flow then means finding a (time-varying) vector field $\mathbf{\hat{v}}$ that is tangent to the surface at all times and satisfies the above equation at every point $x \in \mathcal{M}$ on the moving surface.

Let us conclude this subsection with a remark that, in general, there exist infinitely many parametrisations $\y$ for a given evolving surface.
The actual surface velocity however might be unknown or cannot be estimated from the data, as it is the case in this work.
As a remedy we impose a surface velocity by choosing a natural surface parametrisation $\y$.
For a further discussion of this matter we refer to~\cite{BauGraKir15}.

Moreover, we again stress that the sought tangent vector field $\mathbf{\hat{v}}$ depends on the chosen $\y$, or equivalently on $\mathbf{\hat{V}}$, and should be interpreted with care.
The actual trajectories $\gamma$ though can be reconstructed by finding the integral curves of~\eqref{eq:totalmotion}.
For this precise approach we point the reader to~\cite{KirLanSch15}.

\subsection{Variational Formulation} \label{sec:variationalformulation}

The optical flow equation~\eqref{eq:ofc} derived above is underdetermined and, in general, a unique solution is not ensured.
A common technique to deal with non-uniqueness is Tikhonov regularisation, where one finds a minimiser of
\begin{equation*}
	\norm{d_{t}^{\mathbf{\hat{V}}} \hat{f} + \nabla_{\mathcal{M}} \hat{f} \cdot \mathbf{\hat{v}}}_{L^{2}(\mathcal{M})}^{2} + \alpha \mathcal{R}(\mathbf{\hat{v}}).
\end{equation*}
Here, $\mathcal{R}(\mathbf{\hat{v}})$ is a regularisation functional and $\alpha > 0$ a regularisation parameter, balancing the two terms.
The first term is typically referred to as \emph{data term} whereas the second is called \emph{smoothness term}.
The latter enforces uniqueness and incorporates prior knowledge about favoured solutions.

A common choice for $\mathcal{R}(\mathbf{\hat{v}})$ is the squared $H^{1}$ Sobolev seminorm, involving first derivatives with respect to space and time.
It favours spatial as well as temporal regularity and is of particular interest when trying to estimate trajectories of objects, albeit computationally more demanding.
See, for example~\cite{KirLanSch15, WeiSchn01b} and~\cite{BauGraKir15}.

Alternatively, one can omit temporal regularisation leading to a regulariser of the form
\begin{equation*}
	\mathcal{R}(\mathbf{\hat{v}}) = \int_{I} \norm{\mathbf{\hat{v}}(t, \cdot)}_{H^{1}(\mathcal{M}_{t}, T\mathcal{M}_{t})}^{2} \; dt,
\end{equation*}
which is equivalent to solving for each time instant separately.
It resembles the original formulation in~\cite{HorSchu81} and its extension to 2-manifolds~\cite{LefBai08}.
In the present article we follow this approach and attempt in finding the unique minimiser $\mathbf{\hat{v}} \in T\mathcal{M}_{t}$ of the energy
\begin{equation}
	\mathcal{E}_{\alpha}(\mathbf{\hat{v}}) \coloneqq \norm{d_{t}^{\mathbf{\hat{V}}} \hat{f} + \nabla_{\mathcal{M}} \hat{f} \cdot \mathbf{\hat{v}}}_{L^{2}(\mathcal{M}_{t})}^{2} + \alpha \norm{\mathbf{\hat{v}}}_{H^{1}(\mathcal{M}_{t}, T\mathcal{M}_{t})}^{2}
\label{eq:functional}
\end{equation}
for each time instant $t \in I$ separately.
Superimposing temporal regularisation however is straightforward, see~\cite[Sec.~2.2.2]{KirLanSch15}, but not considered here.
\section{Numerical Solution} \label{sec:numerics}

\subsection{Finite-dimensional Projection}

For the subsequent discussion we let $t \in I$ be arbitrary but fixed and assume to be given a parametrisation $\y(t, \cdot): \Omega \to \mathcal{M}_{t}$ as defined in~\eqref{eq:param}.
We defer the question of how to find it to Sec.~\ref{sec:surface}.
Moreover, for notational convenience, we relabel the set of tangential vector spherical harmonics~\eqref{eq:vspharm} using a single index letter $p \in \mathbb{N}$.
For instance, for the expansion of a tangent vector field on $\mathcal{S}^2$ we simply write $\mathbf{\tilde{u}} = \sum_{p} u_{p} \mathbf{\tilde{y}}_{p}$, where $u_{p} \in \R$ are the coefficients.

We intend to approximate the solution of the problem
\begin{equation*}
	\min_{\mathbf{\hat{v}} \in H^{1}(\mathcal{M}_{t}, T\mathcal{M}_{t})} \mathcal{E}_{\alpha}(\mathbf{\hat{v}})
\end{equation*}
in a finite-dimensional subspace $\mathcal{U} \subset H^{1}(\mathcal{M}_{t}, T\mathcal{M}_{t})$, where $\mathcal{E}_{\alpha}$ is defined in~\eqref{eq:functional}.
We define this space as
\begin{equation*}
	\mathcal{U} = \mathrm{span} \{ \mathbf{\hat{y}}_{p}: p \in J_{\mathcal{U}} \},
\end{equation*}
where $J_{\mathcal{U}} \subset \mathbb{N}$ is a finite index set and $\mathbf{\hat{y}}_{p}$ is the pushforward of a particular vector spherical harmonic $\mathbf{\tilde{y}}_{p}$ via the differential $D\tilde{\phi}$, see~\eqref{eq:differential}.
Figure~\ref{fig:cd} gives a descriptive view of the relation between the introduced spaces and tangent vector fields.

\begin{figure}[t]
	\centering
	\begin{tikzpicture}[every node/.style={on grid, scale=1, node distance=2.5 and 2.5}]
		\node (R2) 							{$\R^{2}$};
		\node (TS2)		[right=of R2]		{$T\mathcal{S}^{2}$};
		\node (TNt)		[right=of TS2]		{$T\mathcal{M}_{t}$};
		\node (Omega)	[below=of R2]		{$\Omega$};
		\node (S2)		[right=of Omega]	{$\mathcal{S}^2$};
		\node (Nt)		[right=of S2]		{$\mathcal{M}_{t}$};
		\draw [->] (R2) -- node [below] {$D\x$} (TS2);
		\draw [->] (TS2) -- node [below] {$D\tilde{\phi}(t, \cdot)$} (TNt);
		\draw [->] (Omega) -- node [above] {$\x$} (S2);
		\draw [->] (S2) -- node [above] {$\tilde{\phi}(t, \cdot)$} (Nt);
		\draw [->] (R2) to [bend left=30] node [above] {$D\y(t, \cdot)$} (TNt);
		\draw [->] (Omega) to [bend right=30] node [below] {$\y(t, \cdot)$} (Nt);
		\draw [->] (Omega) -- node [left] {$\mathbf{y}_{p}$} (R2);
		\draw [->] (S2) -- node [left] {$\mathbf{\tilde{y}}_{p}$} (TS2);
		\draw [->] (Nt) -- node [left] {$\mathbf{\hat{y}}_{p}$} (TNt);
	\end{tikzpicture}
	\caption{Commutative diagram relating spaces $\Omega$, $\mathcal{S}^{2}$, and $\mathcal{M}_{t}$, and tangent vector fields. We highlight that $\mathbf{y}_{p}$ is the coordinate representation, see Sec.~\ref{sec:diffgeo}, of a particular tangential vector spherical harmonic $\mathbf{\tilde{y}}_{p}$ and $\mathbf{\hat{y}}_{p}$ is its uniquely identified tangent vector field on $\mathcal{M}_{t}$.}
	\label{fig:cd}
\end{figure}
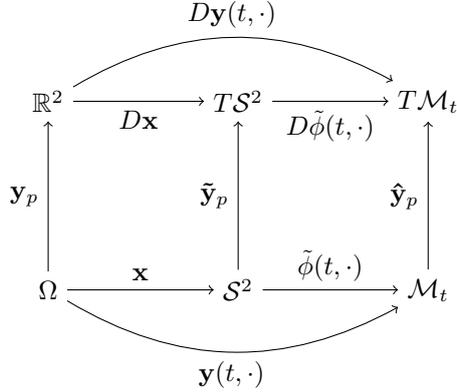

The sought vector field is then uniquely expanded as
\begin{equation}
	\mathbf{\hat{v}} = \sum_{p \in J_{\mathcal{U}}} v_{p} \mathbf{\hat{y}}_{p},
\label{eq:expansion}
\end{equation}
with $v_{p} \in \R$, $p \in J_{\mathcal{U}}$, being the coefficients.
Minimisation of functional~\eqref{eq:functional} results in a finite-dimensional optimisation problem over $\R^{\abs{J_{\mathcal{U}}}}$.
Plugging ansatz \eqref{eq:expansion} into \eqref{eq:functional} and writing out definition \eqref{eq:sobolevnorm} of the Sobolev $H^{1}(\mathcal{M}_{t}, T\mathcal{M}_{t})$ norm gives
\begin{equation}
	\mathcal{E}_{\alpha}(\mathbf{\hat{v}}) = \int_{\mathcal{M}_{t}} \Bigl( \bigl( d_{t}^{\mathbf{\hat{V}}} \hat{f} + \sum_{p \in J_{\mathcal{U}}} v_{p} ( \nabla_{\mathcal{M}} \hat{f} \cdot \mathbf{\hat{y}}_{p} ) \bigr)^{2} + \alpha \norm{\sum_{p \in J_{\mathcal{U}}} v_{p} \nabla \mathbf{\hat{y}}_{p}}_{2}^{2} \Bigr) \; d\mathcal{M}_{t}.
\label{eq:optproblem}
\end{equation}

By applying the definition of the Hilbert-Schmidt norm~\eqref{eq:hsnorm}, using linearity of the covariant derivative $\nabla_{\mathbf{\hat{u}}} \mathbf{\hat{v}}$ with respect to $\mathbf{\hat{v}}$, and the definition of the norm of $\R^{3}$ we obtain the representation
\begin{align*}
	\norm{\nabla \sum_{p \in J_{\mathcal{U}}} v_{p} \mathbf{\hat{y}}_{p}}_{2}^{2} & = \sum_{i=1}^{2} \norm{\sum_{p \in J_{\mathcal{U}}} v_{p} \nabla_{\mathbf{\hat{e}}_{i}} \mathbf{\hat{y}}_{p}}^{2} \\
	& = \sum_{i=1}^{2} \Bigl( \sum_{p \in J_{\mathcal{U}}} v_{p} \nabla_{\mathbf{\hat{e}}_{i}} \mathbf{\hat{y}}_{p} \cdot \sum_{q \in J_{\mathcal{U}}} v_{q} \nabla_{\mathbf{\hat{e}}_{i}} \mathbf{\hat{y}}_{q} \Bigr) \\
	& = \sum_{i=1}^{2} \sum_{p, q \in J_{\mathcal{U}}} v_{p} v_{q} \bigl( \nabla_{\mathbf{\hat{e}}_{i}} \mathbf{\hat{y}}_{p} \cdot \nabla_{\mathbf{\hat{e}}_{i}} \mathbf{\hat{y}}_{q} \bigr)
\end{align*}
for the regularisation term.

The optimality conditions for the discrete minimisation problem~\eqref{eq:optproblem} are obtained by taking $\partial \mathcal{E}_{\alpha} / \partial v_{p} = 0$, for all $p \in J_{\mathcal{U}}$, and are given by
\begin{equation}
\begin{aligned}
	\sum_{q \in J_{\mathcal{U}}} v_{q} \int_{\mathcal{M}_{t}} \Bigl( \bigl( \nabla_{\mathcal{M}} \hat{f} \cdot \mathbf{\hat{y}}_{p} \bigr) \bigl( \nabla_{\mathcal{M}} \hat{f} \cdot \mathbf{\hat{y}}_{q} \bigr) + \alpha \sum_{i=1}^{2} \bigl( \nabla_{\mathbf{\hat{e}}_{i}} \mathbf{\hat{y}}_{p} \cdot \nabla_{\mathbf{\hat{e}}_{i}} \mathbf{\hat{y}}_{q} \bigr) \Bigr) \; d\mathcal{M}_{t} \\
	= - \int_{\mathcal{M}_{t}} d_{t}^{\mathbf{\hat{V}}} \hat{f} \bigl( \nabla_{\mathcal{M}} \hat{f} \cdot \mathbf{\hat{y}}_{p} \bigr) \; d\mathcal{M}_{t}, \quad p \in J_{\mathcal{U}}.
\end{aligned}
\label{eq:optcond}
\end{equation}
In matrix form they read
\begin{equation*}
	(A + \alpha D)v = b,
\end{equation*}
where $v = (v_{1}, \dots, v_{\abs{J_{\mathcal{U}}}})^{\top} \in \R^{\abs{J_{\mathcal{U}}}}$ is the vector of unknowns.
The entries of the matrix $A = (a_{pq})_{pq}$ are
\begin{equation*}
	a_{pq} = \int_{\mathcal{M}_{t}} \bigl( \nabla_{\mathcal{M}} \hat{f} \cdot \mathbf{\hat{y}}_{p} \bigr) \bigl( \nabla_{\mathcal{M}} \hat{f} \cdot \mathbf{\hat{y}}_{q} \bigr) \; d\mathcal{M}_{t},
\end{equation*}
the entries of the matrix $D = (d_{pq})_{pq}$ associated with the regularisation term are given by
\begin{equation*}
	d_{pq} = \int_{\mathcal{M}_{t}} \sum_{i=1}^{2} \bigl( \nabla_{\mathbf{\hat{e}}_{i}} \mathbf{\hat{y}}_{p} \cdot \nabla_{\mathbf{\hat{e}}_{i}} \mathbf{\hat{y}}_{q} \bigr) \; d\mathcal{M}_{t},
\end{equation*}	
and the entries of the vector $b = (b_{p})_{p}$ are
\begin{equation*}
	b_{p} = - \int_{\mathcal{M}_{t}} d_{t}^{\mathbf{\hat{V}}} \hat{f} \bigl( \nabla_{\mathcal{M}} \hat{f} \cdot \mathbf{\hat{y}}_{p} \bigr) \; d\mathcal{M}_{t}.
\end{equation*}

\subsection{Rewriting the Optimality Conditions}

Even tough directly solving the derived optimality conditions~\eqref{eq:optcond} is perfectly legitimate, we take a different approach.
The goal of this section is to rewrite the optimality conditions in terms of quantities defined on the 2-sphere, thereby allowing a more general treatment.
On the one hand, we want to deal with all surfaces $\mathcal{M}_{t}$ for all $t \in I$ in a unified manner and, on the other hand, we aim at evaluating~\eqref{eq:optcond} numerically on the (approximated) sphere without having to deal with multiple charts, see e.g.~\cite{Car76}.

The following is a straight-forward generalisation of~\cite[Lemma 2]{KirLanSch15}.

\begin{lemma}
\label{lem:ofc}
Consider time $t \in I$ arbitrary but fixed.
Let $\mathbf{\tilde{v}} = v^{i} \partial_{i} \x$ and $\mathbf{\hat{v}} = v^{i} \partial_{i} \y$ be two tangent vector fields on $\mathcal{S}^2$ and $\mathcal{M}_{t}$, respectively, such that they are related via the differential~\eqref{eq:differential}.
Then, the parametrised optical flow equation~\eqref{eq:ofc} is equivalent to
\begin{equation*}
	\partial_{t} \tilde{f} + \nabla_{\mathcal{S}^2} \tilde{f} \cdot \mathbf{\tilde{v}} = 0.
\end{equation*}
\end{lemma}

\begin{proof}
According to the definitions~\eqref{eq:timederiv} and~\eqref{eq:f_on_manifold}, we have
\begin{align*}
	d_{t}^{\mathbf{\hat{V}}} \hat{f}(t, \y(t, \xi)) & = \frac{d}{dt} \hat{f}(t, \y(t, \xi)) \\
	& = \frac{d}{dt} \tilde{f}(t, \x(\xi)) \\
	& = \partial_{t} \tilde{f}(t, \x(\xi))
\end{align*}
and it remains to show the identity
\begin{equation*}
	\nabla_{\mathcal{M}} \hat{f} \cdot \mathbf{\hat{v}} = \nabla_{\mathcal{S}^2} \tilde{f} \cdot \mathbf{\tilde{v}},
\end{equation*}
where we have omitted the arguments $(t, \y(t, \xi))$ on the left and $(t, \x(\xi))$ on the right hand side, respectively.
It follows directly from the coordinate representation of the directional derivatives~\eqref{eq:directionalderiv_S} and~\eqref{eq:directionalderiv_N}.
\end{proof}

In order to give coordinate expressions for the terms in~\eqref{eq:optcond} arising from the regularisation term we locally choose an orthonormal frame $\{ \mathbf{\hat{e}}_{1}(t, \xi), \mathbf{\hat{e}}_{2}(t, \xi) \}$ of the tangent space, see~\eqref{eq:onb}.
As a consequence, the sought tangent vector field $\mathbf{\hat{v}}$ can be written as
\begin{equation}
	\mathbf{\hat{v}} = w^{i} \mathbf{\hat{e}}_{i}
\label{eq:onbcoeff}
\end{equation}
for some components $(w^{1}, w^{2})^{\top}$.
The reason for expressing the unknown in an orthonormal frame, rather than the coordinate frame, is to simplify matters with regard to the Hilbert-Schmidt norm~\eqref{eq:hsnorm} of the covariant derivative.

However, the chosen Galerkin method expands the unknown $\mathbf{\hat{v}}$ in terms of the pushfoward of vector fields which are defined on the 2-sphere, cf.~\eqref{eq:expansion}.
We necessarily need to establish the relation between the intended form~\eqref{eq:onbcoeff} and the expression in terms of the coordinate frame.

\begin{lemma}
Again, let $t \in I$ be arbitrary but fixed and let $\mathbf{\tilde{u}} = u^{i} \partial_{i} \x$ be a tangent vector field on $\mathcal{S}^2$.
Then, for a tangent vector field $\mathbf{\hat{v}} = w^{i} \mathbf{\hat{e}}_{i}$ on $\mathcal{M}_{t}$, we have $\mathbf{\hat{v}} = D\tilde{\phi}(\mathbf{\tilde{u}})$ if and only if $w^{i} = (\alpha^{-1})_{\ell}^{i} u^{\ell}$.
\end{lemma}
\begin{proof}
$(\Leftarrow)$ First, note that $\alpha_{i}^{j} (\alpha^{-1})_{\ell}^{i} = \delta_{j\ell}$.
Expanding $\mathbf{\hat{v}}$ gives
\begin{equation*}
	\mathbf{\hat{v}} = w^{i} \mathbf{\hat{e}}_{i} = w^{i} \alpha_{i}^{j} \partial_{j} \y = (\alpha^{-1})_{\ell}^{i} u^{\ell} \alpha_{i}^{j} \partial_{j} \y = u^{j} \partial_{j} \y = D\tilde{\phi}(\mathbf{\tilde{u}}),
\end{equation*}
where we have used~\eqref{eq:Dphiv}, cf. also Fig.~\ref{fig:cd}.

$(\Rightarrow)$ Suppose $\mathbf{\hat{v}} = D\tilde{\phi}(\mathbf{\tilde{u}})$.
Let us take the inner product with $\mathbf{\hat{e}}_{i}$ on both sides.
For the left hand side we have
\begin{equation*}
	\mathbf{\hat{v}} \cdot \mathbf{\hat{e}}_{i} = w^{j} \mathbf{\hat{e}}_{j} \cdot \mathbf{\hat{e}}_{i} = w^{j} \delta_{ji} = w^{i}.
\end{equation*}
For the right hand side we first observe that, by inversion of the matrix $\alpha$ in~\eqref{eq:onb}, it holds that $\partial_{j} \y = (\alpha^{-1})_{j}^{\ell} \mathbf{\hat{e}}_{\ell}$.
Then,
\begin{align*}
	D\tilde{\phi}(\mathbf{\tilde{u}}) \cdot \mathbf{\hat{e}}_{i} & = u^{\ell} \partial_{\ell} \y \cdot \mathbf{\hat{e}}_{i} \\
	& = u^{\ell} (\alpha^{-1})_{\ell}^{j} \mathbf{\hat{e}}_{j} \cdot \mathbf{\hat{e}}_{i} \\
	& = u^{\ell} (\alpha^{-1})_{\ell}^{j} \delta_{ji} \\
	& = u^{\ell} (\alpha^{-1})_{\ell}^{i}
\end{align*}
and we conclude that $w^{i} = u^{\ell} (\alpha^{-1})_{\ell}^{i}$ as required.
\end{proof}

With the above relation at hand we obtain the following form.
\begin{lemma}
\label{lem:ip}
Let $t \in I$ and let $\mathbf{\hat{u}} = u^{i} \partial_{i} \y$ and $\mathbf{\hat{v}} = v^{i} \partial_{i} \y$ be two tangent vector fields on $\mathcal{M}_{t}$.
Then, we have
\begin{equation*}
	\nabla_{\mathbf{\hat{e}}_{i}} \mathbf{\hat{u}} \cdot \nabla_{\mathbf{\hat{e}}_{i}} \mathbf{\hat{v}} = \sum_{j=1}^{2} D_{i} u^{j} D_{i} v^{j},
\end{equation*}
where
\begin{equation*}
	D_{i} u^{j} \coloneqq \alpha_{i}^{k} \partial_{k} \bigl( (\alpha^{-1})_{\ell}^{j} u^{\ell} \bigr) + (\alpha^{-1})_{\ell}^{k} u^{\ell} \hat{\Gamma}_{ik}^{j}, \quad i, j = \{1, 2\},
\end{equation*}
and $D_{i} v^{j}$ are defined accordingly.
\end{lemma}
$\hat{\Gamma}_{ik}^{j}$ denote the Christoffel symbols with regard to the orthonormal frame $\{ \mathbf{\hat{e}}_{1}, \mathbf{\hat{e}}_{2} \}$ and are defined as
\begin{equation}
	\nabla_{\mathbf{\hat{e}}_{i}} \mathbf{\hat{e}}_{k} = \hat{\Gamma}_{ik}^{j} \mathbf{\hat{e}}_{j}.
\label{eq:gamma}
\end{equation}
We refer to~\cite[Lemma~3]{KirLanSch15} for their derivation.

\begin{proof}
First let us show that, for $\mathbf{\hat{u}} = w^{j} \mathbf{\hat{e}}_{j}$ as in~\eqref{eq:onbcoeff}, it holds that
\begin{equation*}
	\nabla_{\mathbf{\hat{e}}_{i}} \mathbf{\hat{u}} = D_{i} u^{j} \mathbf{\hat{e}}_{j}.
\end{equation*}			
By the product rule for the covariant derivative~\eqref{eq:covderiv},
\begin{equation}
	\nabla_{\mathbf{\hat{e}}_{i}} w^{j} \mathbf{\hat{e}}_{j} = \mathbf{\hat{e}}_{j} \nabla_{\mathbf{\hat{e}}_{i}} w^{j} + w^{j} \nabla_{\mathbf{\hat{e}}_{i}} \mathbf{\hat{e}}_{j}.
\label{eq:covderivprodrule}
\end{equation}
Consider the first term of the sum and let $\mathbf{\hat{e}}_{i}$ be represented in the coordinate basis as in~\eqref{eq:onb}.
Then,
\begin{equation*}
	\nabla_{\mathbf{\hat{e}}_{i}} w^{j} = \nabla_{\alpha_{i}^{k} \partial_{k} \y} w^{j}.
\end{equation*}
Linearity of the lower argument of the covariant derivative with respect to $C^{\infty}(\mathcal{M}_{t})$ functions, cf.~\eqref{eq:directionalderiv_N}, yields
\begin{equation*}
	\nabla_{\alpha_{i}^{k} \partial_{k} \y} w^{j} = \alpha_{i}^{k} \nabla_{\partial_{k} \y} w^{j}
\end{equation*}
and by realising that $\nabla_{\partial_{k} \y} w^{j}$ is just the directional derivative~\eqref{eq:directionalderiv_N} along $\partial_{k} \y$ we obtain
\begin{equation*}
	\alpha_{i}^{k} \nabla_{\partial_{k} \y} w^{j} = \alpha_{i}^{k} \partial_{k} w^{j}.
\end{equation*}

Moreover, in the second term of the sum in~\eqref{eq:covderivprodrule} we use definition~\eqref{eq:gamma}.
Thus, by summing up all terms in~\eqref{eq:covderivprodrule} we obtain
\begin{equation*}
	\nabla_{\mathbf{\hat{e}}_{i}} w^{j} \mathbf{\hat{e}}_{j} = \bigl( \alpha_{i}^{k} \partial_{k} w^{j} + w^{j} \hat{\Gamma}_{ik}^{j} \bigr) \mathbf{\hat{e}}_{j}.
\end{equation*}
Applying the previous lemma gives coefficients $D_{i} u^{j}$ and $D_{i} v^{j}$ in the intended form.
Finally, it remains to observe that
\begin{align*}
	\nabla_{\mathbf{\hat{e}}_{i}} \mathbf{\hat{u}} \cdot \nabla_{\mathbf{\hat{e}}_{i}} \mathbf{\hat{v}} & = D_{i} u^{j} \mathbf{\hat{e}}_{j} \cdot D_{i} v^{j} \mathbf{\hat{e}}_{j} \\
	& = \sum_{j=1}^{2} D_{i} u^{j} D_{i} v^{j},
\end{align*}
since by definition $\mathbf{\hat{e}}_{i} \cdot \mathbf{\hat{e}}_{j} = \delta_{ij}$.
\end{proof}

Finally, by combining Lemmas~\ref{lem:surfintegral}, \ref{lem:ofc}, and~\ref{lem:ip} we are able to express the optimality conditions~\eqref{eq:optcond} in terms of integrals on the 2-sphere.
Thus, we arrive at the optimality conditions
\begin{equation}
\begin{aligned}
	\sum_{q \in J_{\mathcal{U}}} v_{q} \int_{\mathcal{S}^2} \Bigl( \bigl( \nabla_{\mathcal{S}^2} \tilde{f} \cdot \mathbf{\tilde{y}}_{p} \bigr) \bigl( \nabla_{\mathcal{S}^2} \tilde{f} \cdot \mathbf{\tilde{y}}_{q} \bigr) + \alpha \sum_{i,j=1}^{2} D_{i} y_{p}^{j} D_{i} y_{q}^{j} \Bigr) \; \tilde{\rho} \sqrt{\norm{\nabla_{\mathcal{S}^2} \tilde{\rho}}^{2} + \tilde{\rho}^{2}} \, d\mathcal{S}^2 \\
	= - \int_{\mathcal{S}^2} \partial_{t} \tilde{f} \bigl( \nabla_{\mathcal{S}^2} \tilde{f} \cdot \mathbf{\tilde{y}}_{p} \bigr) \; \tilde{\rho} \sqrt{\norm{\nabla_{\mathcal{S}^2} \tilde{\rho}}^{2} + \tilde{\rho}^{2}} \, d\mathcal{S}^2, \quad p \in J_{\mathcal{U}},
\end{aligned}
\label{eq:optcondpullback}
\end{equation}
where $\tilde{\rho} \sqrt{\norm{\nabla_{\mathcal{S}^2} \tilde{\rho}}^{2} + \tilde{\rho}^{2}}$ arises from the Jacobian~\eqref{eq:jacobian}, see also Lemma~\ref{lem:surfintegral}.

The entries of the matrices $A$, $D$ and of the vector $b$, respectively, are then given by
\begin{equation}
	a_{pq} = \int_{\mathcal{S}^2} \bigl( \nabla_{\mathcal{S}^2} \tilde{f} \cdot \mathbf{\tilde{y}}_{p}) (\nabla_{\mathcal{S}^2} \tilde{f} \cdot \mathbf{\tilde{y}}_{q} \bigr) \; \tilde{\rho} \sqrt{\norm{\nabla_{\mathcal{S}^2} \tilde{\rho}}^{2} + \tilde{\rho}^{2}} \, d\mathcal{S}^2,
\label{eq:a}
\end{equation}
\begin{equation}
	d_{pq} = \int_{\mathcal{S}^2} \sum_{i,j=1}^{2} D_{i} y_{p}^{j} D_{i} y_{q}^{j} \; \tilde{\rho} \sqrt{\norm{\nabla_{\mathcal{S}^2} \tilde{\rho}}^{2} + \tilde{\rho}^{2}} \, d\mathcal{S}^2,
\label{eq:d}
\end{equation}	
and
\begin{equation}
	b_{p} = - \int_{\mathcal{S}^2} \partial_{t} \tilde{f} \bigl( \nabla_{\mathcal{S}^2} \tilde{f} \cdot \mathbf{\tilde{y}}_{p} \bigr) \; \tilde{\rho} \sqrt{\norm{\nabla_{\mathcal{S}^2} \tilde{\rho}}^{2} + \tilde{\rho}^{2}} \, d\mathcal{S}^2.
\label{eq:b}
\end{equation}

\subsection{Surface Parametrisation} \label{sec:surface}

In order to actually compute the above optimality conditions it remains to determine the radius $\tilde{\rho}: I \times \mathcal{S}^2 \to (0, \infty)$ in the presumed parametrisation~\eqref{eq:param}.
Again, we continue the discussion for one particular but fixed time $t \in I$ and drop the argument whenever convenient.

Estimating $\tilde{\rho}(t, \cdot): \mathcal{S}^2 \to (0, \infty)$ is closely related to surface interpolation from scattered data.
Given noisy data $\tilde{\rho}^{\delta}$ and a parameter $\beta > 0$, it amounts to finding the unique minimiser of the functional
\begin{equation}
	\mathcal{F}_{\beta}(\tilde{\rho}) \coloneqq \norm{\tilde{\rho} - \tilde{\rho}^{\delta}}_{L^{2}(\mathcal{S}^2)}^{2} + \beta \abs{\tilde{\rho}}_{H^{s}(\mathcal{S}^2)}^{2},
\label{eq:surffunctional}
\end{equation}
where $s > 0$ is a sufficiently large real number, cf. definition~\eqref{eq:sobolevseminorm}.
The first term penalises deviation from the observed data whereas the second term enforces spatial regularity of the solution.

In practice, however, $N > 0$ evaluations $\{ \tilde{\rho}^{\delta}(x_{i}): x_{i} \in \mathcal{S}^2 \}_{i=1}^{N}$ are given at pairwise distinct points on the 2-sphere.
In our particular application these correspond to taking the norm in $\R^{3}$ of pairwise distinct sampling points lying on the sphere-like surface $\mathcal{M}_{t}$:
\begin{equation}
	\tilde{\rho}^{\delta}(\bar{x}_{i}) = \norm{x_{i}}, \; x_{i} \in \R^{3} \setminus \{0\}, \; i=1, \dots, N,
\label{eq:pointeval}
\end{equation}
where $\bar{x}_{i} = x_{i} / \norm{x_{i}}$ is the radial projection onto $\mathcal{S}^2$.
We again point the reader to Fig.~\ref{fig:surfaces}.

Furthermore, before turning to the numerical solution of~\eqref{eq:surffunctional}, let us briefly discuss the regularity requirements.
In~\cite{BauGraKir15}, the authors demand twice continuous differentiability for both the manifold $\mathcal{M}_{t}$ and the map $\y(t, \cdot)$ to obtain well-posedness of the optical flow problem.
By definition of the parametrisation~\eqref{eq:param} we require that $\tilde{\rho}(t, \cdot) \in C^{2}(\mathcal{S}^2)$.
As a consequence of Theorem~2.7 in~\cite[Chapter~2.6]{Heb99} regarding Sobolev embeddings, the space $H^{s}(\mathcal{S}^2)$ for $s > 3$ is the appropriate choice, i.e. $H^{s}(\mathcal{S}^2) \subset C^{2}(\mathcal{S}^2)$.

Numerically, we approximate the solution of the problem
\begin{equation*}
	\min_{\tilde{\rho} \in H^{s}(\mathcal{S}^2)} \mathcal{F}_{\beta}(\tilde{\rho})
\end{equation*}
by considering a finite-dimensional subspace $\mathcal{Q} \subset H^{s}(\mathcal{S}^2)$ and point evaluations~\eqref{eq:pointeval}.
In contrast to above, the space
\begin{equation*}
	\mathcal{Q} = \mathrm{span} \{ \tilde{Y}_{p}: p \in J_{\mathcal{Q}} \},
\end{equation*}
where $J_{\mathcal{Q}} \subset \mathbb{N}_{0}$ again is an index set, is spanned by scalar spherical harmonics.
The sought function is expanded as
\begin{equation*}
	\tilde{\rho} = \sum_{p \in J_{\mathcal{Q}}} \rho_{p} \tilde{Y}_{p},
\end{equation*}
where the unknowns are the coefficients $\rho_{p} \in \R$, for $p \in J_{\mathcal{Q}}$.
Plugging into~\eqref{eq:surffunctional}, applying definition~\eqref{eq:sobolevseminorm}, and taking $\partial \mathcal{F} / \partial \rho_{p}$, for all $p \in J_{\mathcal{Q}}$, gives the optimality conditions
\begin{equation}
	\sum_{q \in J_{\mathcal{Q}}} \rho_{q} \Bigl( \sum_{i=1}^{N} \tilde{Y}_{p}(\bar{x}_{i}) \tilde{Y}_{q}(\bar{x}_{i}) \Bigr) + \beta \lambda_{p}^{s} \rho_{p} = \sum_{i=1}^{N} \norm{x_{i}} \tilde{Y}_{p}(\bar{x}_{i}), \quad p \in J_{\mathcal{Q}}.
\label{eq:surfoptcond}
\end{equation}

Denoting by $\varrho = (\rho_{1}, \dots, \rho_{\abs{J_{\mathcal{Q}}}})^{\top} \in \R^{\abs{J_{\mathcal{Q}}}}$ the vector of unknown coefficients, the equations~\eqref{eq:surfoptcond} can be written in matrix-vector form as
\begin{equation*}
	(L + \beta M)\varrho = c,
\end{equation*}
The entries of the matrix $L = (l_{pq})_{pq}$ are
\begin{equation*}
	l_{pq} = \sum_{i=1}^{N} \tilde{Y}_{p}(\bar{x}_{i}) \tilde{Y}_{q}(\bar{x}_{i}),
\end{equation*}
the matrix $M = \mathrm{diag}(\lambda_{1}^{s}, \dots, \lambda_{\abs{J_{\mathcal{Q}}}}^{s})$ is a diagonal matrix, and
\begin{equation*}
	c_{p} = \sum_{i=1}^{N} \norm{x_{i}} \tilde{Y}_{p}(\bar{x}_{i}).
\end{equation*}

\subsection{Numerical Approximation} \label{sec:numericalapprox}

Let us finally discuss the numerical solution of the optimality conditions \eqref{eq:optcondpullback}.
In particular, one needs to (approximately) evaluate the integrals \eqref{eq:a}, \eqref{eq:d}, and \eqref{eq:b}.
Even though integrals on the 2-sphere can be computed exactly and quadrature rules exist up to a certain degree, see e.g.~\cite{AtkHan12, HesSloWom10}, we instead prefer to use a triangulation together with an appropriate quadrature.
The reason is that numerical quadrature on the sphere would have to be of rather high degree to reproduce small details and features of the data, contrary to the chosen quadrature, which can easily be refined up to the desired precision.
Finally let us mention that, for a more accurate evaluation of the integrals, one can introduce an intermediate (radial) map from the polyhedron to geodesic triangles.
See e.g.~\cite[Sec.~7.2]{HesSloWom10}.

We use a polyhedral approximation $\mathcal{S}_{h}^{2} = (\mathcal{V}, \mathcal{T})$ of the 2-sphere $\mathcal{S}^2$.
It is defined by a set $\mathcal{V} = \{v_{1}, \dots, v_{n}\} \subset \mathcal{S}^2$ of vertices and a set $\mathcal{T} = \{T_{1}, \dots, T_{m}\} \subset \mathcal{V} \times \mathcal{V} \times \mathcal{V}$ of triangular faces.
Each triangle is most easily parametrised using barycentric coordinates, see e.g.~\cite[Chapter~5]{BotKobPauAllLev10}.
We associate with each triangle $T_{i} \in \mathcal{T}$ a tuple $(i_{1}, i_{2}, i_{3})$ identifying the corresponding vertices $(v_{i_{1}}, v_{i_{2}}, v_{i_{3}})$, which are arranged in clockwise order.
The parametrisation~\eqref{eq:sphereparam} then reads
\begin{equation*}
	\x_{i}(\xi) = v_{i_{1}} + \xi^{1} (v_{i_{3}} - v_{i_{1}}) + \xi^{2} (v_{i_{2}} - v_{i_{1}})
\end{equation*}
with
\begin{equation*}
	\Omega = \{ \xi \in \R^{2} : \xi^{1} \in [0, 1] \text{ and } \xi^{2} \in [0, 1 - \xi^{1}] \},
\end{equation*}
which is referred to as the reference triangle.
The gradient matrix of $T_{i}$ is then simply
\begin{equation*}
	D\x_{i} = \begin{pmatrix}
		\partial_{1} \x_{i} & \partial_{2} \x_{i}
	\end{pmatrix} =
	\begin{pmatrix}
		v_{i_{3}} - v_{i_{1}} & v_{i_{2}} - v_{i_{1}}
	\end{pmatrix}.
\end{equation*}
The surface normal is constant on $T_{i}$ and is denoted by $\mathbf{\tilde{N}}_{i}$.

We approximate all functions on $\mathcal{S}^2$ by corresponding functions on the polyhedron $\mathcal{S}_{h}^{2}$.
A continuous function $\tilde{f}: \mathcal{S}^2 \to \R$ is replaced by its piecewise polynomial interpolation $\tilde{f}_{h}: \mathcal{S}_{h}^{2} \to \R$ on $\mathcal{S}_{h}^{2}$.
We define it as
\begin{equation}
	\tilde{f}_{h}(\cdot) \coloneqq \sum_{j=1}^{N_{h}} \tilde{\bar{f}}(v_{j}) \tilde{\varphi}_{j}(\cdot).
\label{eq:approx}
\end{equation}
Here, $\{ \tilde{\varphi}_{j} \}$ are $N_{h} = 6$ quadratic shape functions forming a nodal basis together with nodal points $\{ v_{j} \} \subset \mathcal{S}_{h}^{2}$ and $\tilde{\bar{f}}$ is the usual radially constant extension, cf.~\eqref{eq:extension} in Sec.~\ref{sec:diffgeo}.
In other words, $\tilde{f}_{h}$ is both a radial projection from the 2-sphere to the polyhedron $\mathcal{S}_{h}^{2}$ and to piecewise quadratic functions.
Note that the shape functions are defined on the triangular faces $T_{i}$.
Whenever a function $\tilde{f}$ has a dependence on time we simply compute its approximation $\tilde{f}_{h}$ separately for all times $t \in I$.
We point the reader to Fig.~\ref{fig:interp} for a figurative illustration.

\newcommand{\trianglepath}{(canvas polar cs:angle=80,radius=4cm) -- (canvas polar cs:angle=10,radius=4cm) -- (canvas polar cs:angle=40,radius=2.5cm) -- (canvas polar cs:angle=80,radius=4cm)}

\newcommand{\interptrianglepath}{(canvas polar cs:angle=80,radius=7cm) .. controls (canvas polar cs:angle=60,radius=7cm) and (canvas polar cs:angle=55,radius=6.25cm) .. (canvas polar cs:angle=40,radius=5.25cm) .. controls (canvas polar cs:angle=30,radius=6cm) and (canvas polar cs:angle=25,radius=6.5cm) .. (canvas polar cs:angle=10,radius=7cm) .. controls (canvas polar cs:angle=30,radius=7cm) and (canvas polar cs:angle=55,radius=7.5cm) .. (canvas polar cs:angle=80,radius=7cm)}

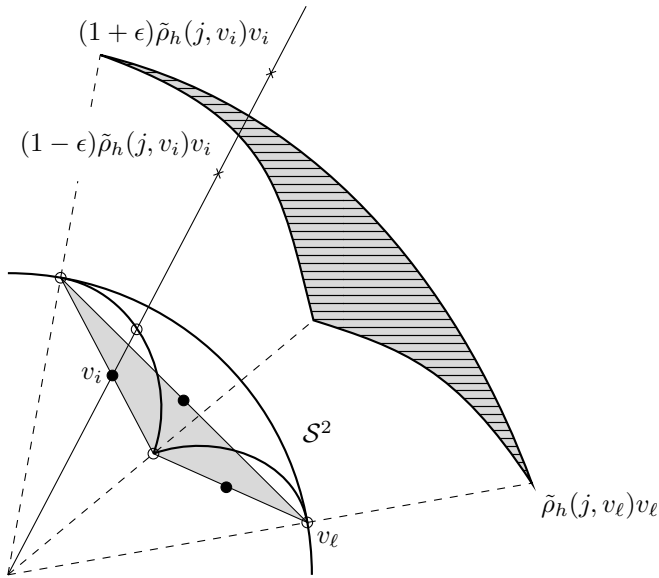
\begin{figure}[t]
	\begin{center}
	\begin{tikzpicture}
		\draw [thick] (0, 4) .. controls (2.22, 4) and (4, 2.22) .. (4, 0);
		\fill [gray!30] \trianglepath;
		\draw [thin] \trianglepath;
		\draw [dashed] (0, 0) to (canvas polar cs:angle=80,radius=5.3cm);
		\draw [dashed] (canvas polar cs:angle=80,radius=6.3cm) to (canvas polar cs:angle=80,radius=7cm);
		\draw [dashed] (0, 0) to (canvas polar cs:angle=40,radius=5.25cm);
		\draw [dashed] (0, 0) to (canvas polar cs:angle=10,radius=7cm);
		\fill [gray!30, postaction={pattern=horizontal lines}] \interptrianglepath;
		\draw [thick] \interptrianglepath;
		\draw [thick] (canvas polar cs:angle=80,radius=4cm) to [out=-10, in=70] (canvas polar cs:angle=40,radius=2.5cm);
		\draw [thick] (canvas polar cs:angle=40,radius=2.5cm) to [out=20, in=100] (canvas polar cs:angle=10,radius=4cm);
		\draw [black] (canvas polar cs:angle=80,radius=4cm) circle (2pt);
		\draw [black] (canvas polar cs:angle=40,radius=2.5cm) circle (2pt);
		\draw [black] (canvas polar cs:angle=10,radius=4cm) circle (2pt);
		\filldraw [black] (canvas polar cs:angle=45,radius=3.27cm) circle (2pt);
		\filldraw [black] (canvas polar cs:angle=22,radius=3.1cm) circle (2pt);
		\filldraw [black] (canvas polar cs:angle=62.5,radius=2.98cm) circle (2pt);
		\draw (0, 0) to (canvas polar cs:angle=62.5,radius=8.5cm);
		\draw [black] (canvas polar cs:angle=62.5,radius=3.67cm) circle (2pt);
		\tikzset{cross/.style={cross out, draw, minimum size=0.1cm, inner sep=0pt, outer sep=0pt}}
		\draw (canvas polar cs:angle=62.5,radius=6cm) node [cross, black, rotate=-25] {};
		\draw (canvas polar cs:angle=62.5,radius=7.5cm) node [cross, black, rotate=-25] {};
		\node [left] at (canvas polar cs:angle=62.5,radius=2.98cm) {$v_{i}$};
		\node [above left] at (canvas polar cs:angle=62.5,radius=6.1cm) {$(1-\epsilon) \tilde{\rho}_{h}(j, v_{i}) v_{i}$};
		\node [above left] at (canvas polar cs:angle=62.5,radius=7.75cm) {$(1+\epsilon) \tilde{\rho}_{h}(j, v_{i}) v_{i}$};
		\node at (canvas polar cs:angle=25,radius=4.5cm) {$\mathcal{S}^{2}$};
		\node [below right] at (canvas polar cs:angle=10,radius=4cm) {$v_{\ell}$};
		\node [below right] at (canvas polar cs:angle=10,radius=7cm) {$\tilde{\rho}_{h}(j, v_{\ell}) v_{\ell}$};
	\end{tikzpicture}
	\end{center}
	\caption{Illustration of a triangular face (filled gray) intersecting the sphere $\mathcal{S}^{2}$ at the vertices (hollow circles). The six nodal points consist of the vertices of the triangle together the edge midpoints (filled black dots). The approximated sphere-like surface is shown by the hatched gray area. A radial line passing through the vertex $v_{i}$ is shown. The hollow circle indicates the intersection with $\mathcal{S}^{2}$ at which $\tilde{\bar{f}}(v_{i})$ in~\eqref{eq:approx} is taken. $\tilde{\bar{f}}$ itself, as described in Sec.~\ref{sec:acquisition}, is assigned by taking the maximum image intensity along the drawn radial line between the two cross marks.}
	\label{fig:interp}
\end{figure}

In further consequence, the fully normalised scalar spherical harmonics, which were introduced in~\eqref{eq:spharm}, are substituted with their corresponding approximations on $\mathcal{S}_{h}^{2}$.
For $\tilde{Y} \in \mathrm{Harm}_{n}$, $n \in \mathbb{N}_{0}$ , we have
\begin{equation}
	\tilde{Y}_{h}(\cdot) = \sum_{j = 1}^{N_{h}} \tilde{\bar{Y}}(v_{j}) \tilde{\varphi}_{j}(\cdot).
\label{eq:spharminterp}
\end{equation}

We chose piecewise quadratic approximations for $\tilde{Y}$ so that we can adequately apply $\nabla_{\mathcal{S}_{h}^{2}}$ and obtain piecewise linear vector fields.
Accordingly, we define approximations of the vector spherical harmonics, introduced in~\eqref{eq:fullynormalisedvspharm}, as follows.

\begin{proposition} \label{def:vspharm}
Let $\tilde{Y} \in \mathrm{Harm}_{n}$, $n \in \mathbb{N}$.
The piecewise linear interpolations of the corresponding tangential vector spherical harmonics on a triangular face $T_{i} \in \mathcal{T}$ are
\begin{align}
	\mathbf{\tilde{y}}_{h}^{(2)}(\x_{i}(\xi)) & = \lambda_{n}^{-1/2} \sum_{j = 1}^{N_{h}} \tilde{\bar{Y}}(v_{j}) \nabla_{\mathcal{S}_{h}^{2}} \tilde{\varphi}_{j}(\x_{i}(\xi)), \label{eq:vspharm1} \\
	\mathbf{\tilde{y}}_{h}^{(3)}(\x_{i}(\xi)) & = \frac{\lambda_{n}^{-1/2}}{2 \abs{T_{i}}} \sum_{j = 1}^{N_{h}} \tilde{\bar{Y}}(v_{j}) \bigl( \partial_{2} \varphi_{j}(\xi) \partial_{1} \x_{i}(\xi) - \partial_{1} \varphi_{j}(\xi) \partial_{2} \x_{i}(\xi) \bigr). \label{eq:vspharm2}
\end{align}
\end{proposition}
Their derivation is deferred to the appendix.

Without loss of generality, let $\tilde{f}_{h}(0, \cdot)$ and $\tilde{f}_{h}(1, \cdot)$ be the approximations of the data $\tilde{f}$ at two subsequent frames.
We define the derivative with respect to time by the forward difference
\begin{equation*}
	\partial_{t} \tilde{f}_{h}(\cdot) \coloneqq \tilde{f}_{h}(1, \cdot) - \tilde{f}_{h}(0, \cdot).
\end{equation*}
Moreover, we replace the surface gradient $\nabla_{\mathcal{S}^2} \tilde{f}$ of a function on $\mathcal{S}^2$ with its counterpart $\nabla_{\mathcal{S}_{h}^{2}} \tilde{f}_{h}$ on $\mathcal{S}_{h}^{2}$, which is computed according to~\eqref{eq:surfgrad_S_coord}.
The function $\tilde{\rho}$ is obtained by solving~\eqref{eq:surfoptcond} and, for numerical computations, is further replaced with its piecewise quadratic interpolation $\tilde{\rho}_{h}$ as in~\eqref{eq:approx}.
Coefficients $\alpha_{i}^{j}$ are computed by the Gram-Schmidt process at the nodal points.
For numerical computations piecewise quadratic approximations, as defined in~\eqref{eq:approx}, are used.

Finally, for the calculation of the integrals we employ the standard quadrature on triangulated spheres, see e.g.~\cite{AtkHan12, HesSloWom10}.
Let $\xi_{c} = (1/3, 1/3)^{\top}$ be the centroid of the reference triangle $\Omega$.
Then, we approximate the spherical integral over a function $\tilde{f}: \mathcal{S}^2 \to \R$ on the 2-sphere by
\begin{equation*}
	\int_{\mathcal{S}^2} \tilde{f} \; d\mathcal{S}^2 \approx \int_{\mathcal{S}_{h}^{2}} \tilde{f}_{h} \; d\mathcal{S}_{h}^{2} \approx \sum_{i=1}^{m} \abs{T_{i}} \tilde{f}_{h}(\x_{i}(\xi_{c})).
\end{equation*}
\section{Experiments} \label{sec:experiments}

\subsection{Microscopy Data} \label{sec:data}

The present data consist of volumetric time-lapse (4-dimensional) images of a live zebrafish embryo during the gastrula period.
These videos were recorded approximately five to ten hours after fertilisation by means of confocal laser-scanning microscopy and feature endodermal cells expressing a green fluorescence protein.
As a consequence, these labelled cells are recorded without background and allow for a separate treatment.
We refer the reader to~\cite{KimBalKimUllSchi95} for many illustrations and a detailed discussion of the zebrafish's developmental process.
Regarding the imaging techniques used during data acquisition we refer to~\cite{MegFra03} and for the treatment of the specimen we point the reader to~\cite{MizVerHeaKurKik08}.

The crucial feature of endodermal cells is the fact that they form a so-called monolayer during early morphogenesis, see~\cite{WarNus99}.
Essentially, it means that the labelled cells do not sit on top of each other but float side by side forming an artificial sphere-shaped layer.
It can be regarded as a surface and allows for the straightforward extraction of an image sequence.
Clearly, this surface is subject to geometric approximations.
For instance, in~\cite{KirLanSch14, SchmShaScheWebThi13} it is assumed an ideal sphere, whereas in~\cite{BauGraKir15} and~\cite{KirLanSch15} only a fraction of the data is considered and modelled as a moving manifold and a height field, respectively, both possessing a boundary.

The recorded data features a cuboid region of approximately $860 \times 860 \times 320 \, \mu \mathrm{m}^{3}$ of the animal hemisphere.
The spatial resolution is $512 \times 512 \times 44$~voxels and the recorded image intensities are in the range $\{ 0, \dots, 255 \}$.
Our sequence contains 75 images with a temporal interval of $240 \, \mathrm{s}$.
For the further discussion, we denote the data by
\begin{equation*}
	f^{\delta} \in \{ 0, \dots, 255 \}^{75 \times 512 \times 512 \times 44}.
\end{equation*}

\subsection{Preprocessing and Surface Data Acquisition} \label{sec:acquisition}

Let us briefly discuss the preprocessing steps required to obtain an image sequence together with the evolving surface.
We limit our consideration to two consecutive frames and denote the respective volumetric data by $f_{0}^{\delta}$ and $f_{1}^{\delta}$.

For each frame, the approximate surface is found by minimising the functional \eqref{eq:surffunctional} with approximate cell centres acting as sample points.
They appear as local maxima in image intensity and are readily located by Gaussian filtering followed by plain thresholding.
However, beforehand the points are centred around the origin by first fitting a sphere and subsequently subtracting the spherical centre.

The triangle mesh $\mathcal{S}_{h}^{2}$ is obtained by iterative refinement of an icosahedron that is inscribed in the 2-sphere, see e.g.~\cite[Chapter~1.3.3]{BotKobPauAllLev10}.
Every refinement step halves the edge lengths by connecting the edge midpoints and projecting them to the unit sphere.
Consequentially, every triangular face is split into four smaller triangles and the total number of faces after $k \in \mathbb{N}_{0}$ subdivisions is $20 \cdot 4^{k}$.
In our case, $k = 7$ refinements are required to resolve the data adequately.

It remains to discuss the acquisition of the approximations $\tilde{f}_{h}(0, \cdot)$ and $\tilde{f}_{h}(1, \cdot)$ on the polyhedron.
For a frame $j \in \{ 0, 1 \}$, we define the value at a nodal point $v_{i} \in \mathcal{S}_{h}^{2}$ in~\eqref{eq:approx} via the projection
\begin{equation*}
	\tilde{\bar{f}}(j, v_{i}) \coloneqq \max_{c \in [1-\varepsilon, 1+\varepsilon]} \mathring{f_{j}^{\delta}}(c \tilde{\rho}_{h}(j, v_{i}) v_{i}),
\end{equation*}
where $\varepsilon > 0$ is chosen sufficiently large.
$\mathring{f_{j}^{\delta}}$ denotes the piecewise linear extension of $f_{j}^{\delta}$ to $\R^{3}$, which is necessary for gridded data.
The above projection within the narrow band
\begin{equation*}
	[(1-\varepsilon) \tilde{\rho}_{h}(j, v_{i}) v_{i}, (1+\varepsilon) \tilde{\rho}_{h}(j, v_{i}) v_{i}]
\end{equation*}
corrects for small deviations of the cells from the fitted surface.
Again, we refer to Fig.~\ref{fig:interp} for illustration.
Finally, all intensities are scaled to the interval $[0, 1]$.
Figure~\ref{fig:data} shows two frames of the extracted image sequence defined on the sphere-like evolving surface.
Figure~\ref{fig:data2} depicts the same matter but in a top view.
For better illustration we have added an artificial mesh.
Its radius has been widened by one percent.

\begin{figure}[t]
	\includegraphics[width=0.49\textwidth]{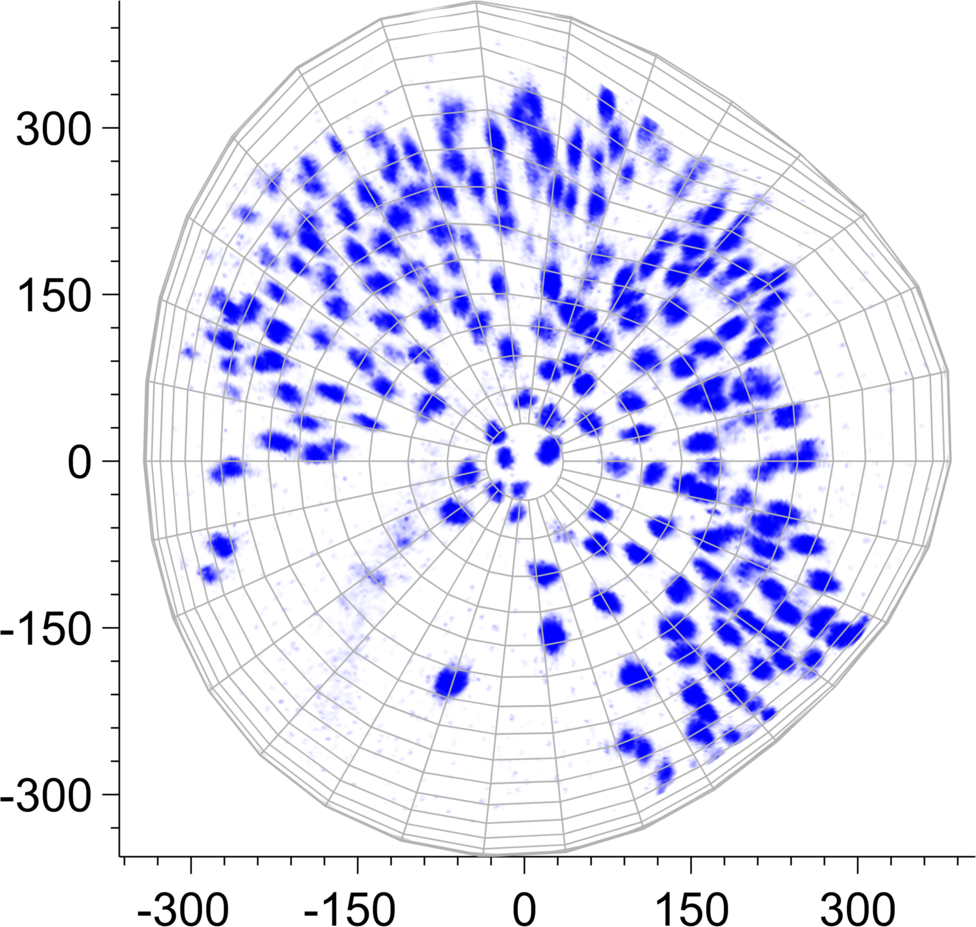} \hfill
	\includegraphics[width=0.49\textwidth]{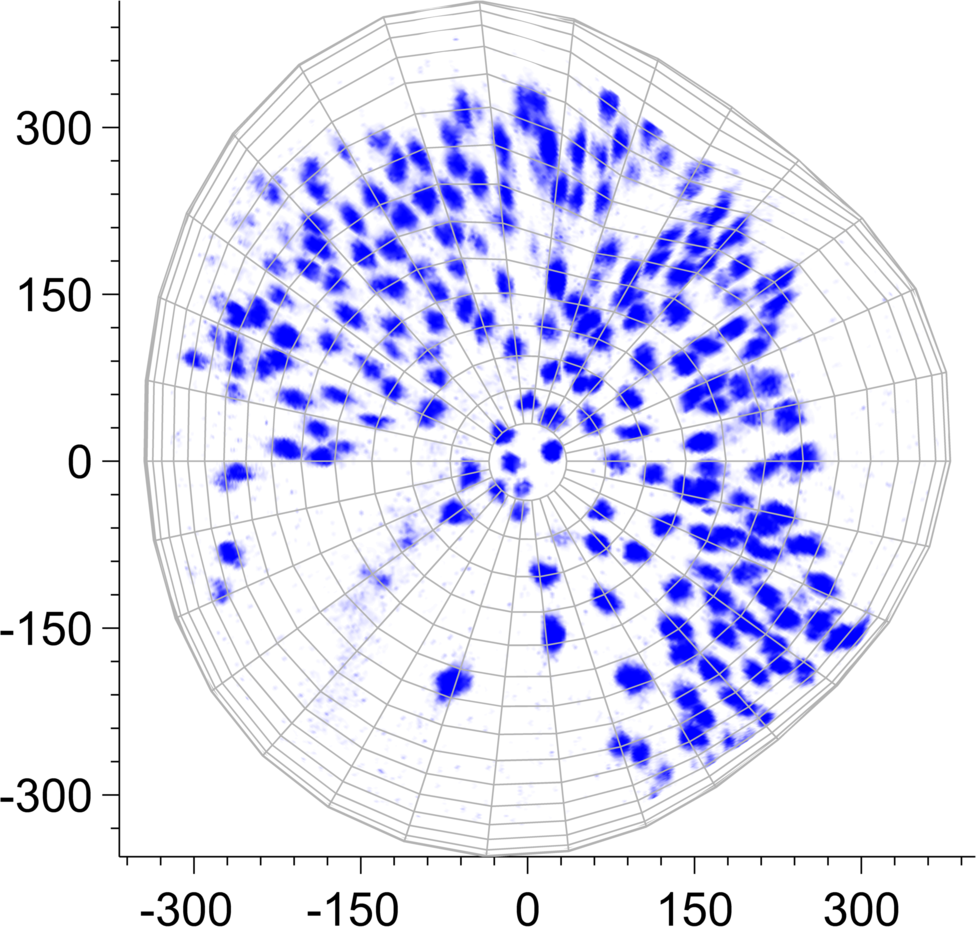}
	\caption{Frames no. 70 (left) and 71 (right) of the processed image sequence in a top view. The embryo's body axis is oriented from bottom left to top right.}
	\label{fig:data2}
\end{figure}

\subsection{Visualisation of Results} \label{sec:visualisation}

We employ the standard flow colour-coding~\cite{BakSchaLewRotBla11} for the visualisation of the computed vector fields.
Its purpose is to create a colour image by assigning every vector a colour from a pre-defined colour disk.
The colour associated is determined by a vector's angle and its length.

However, it was originally defined for planar vector fields and requires adaptation to our particular purpose of tangent vector field visualisation.
To this end, we follow the idea developed in~\cite{KirLanSch14} by first projecting each vector to the plane and then rescaling its length.
Let us denote by $\mathrm{P}_{x_{3}}: (x_{1}, x_{2}, x_{3})^{\top} \mapsto (x_{1}, x_{2}, 0)^{\top}$ the orthogonal projector of $\R^{3}$ onto the $x_{1}$-$x_{2}$-plane.
For a tangent vector field $\mathbf{\hat{v}}$ we apply the colour-coding to the planar vector field
\begin{equation*}
	\frac{\norm{\mathbf{\hat{v}}}}{\norm{\mathrm{P}_{x_{3}} \mathbf{\hat{v}}}} \mathrm{P}_{x_{3}} \mathbf{\hat{v}}.
\end{equation*}
It is constructed so that the length of individual vectors is preserved.
Subsequently, the obtained colour image is mapped back onto the surface.
Clearly, in the above construction, one has to distinguish the cases where $x_{3} \ge 0$ and $x_{3} < 0$. Moreover, $\mathrm{P}_{x_{3}}$ is required to be injective in either case.

The radius $R$ of the colour disk is chosen to be equal to the longest vector in the respective vector field we attempt to visualise.
Table~\ref{tab:radii} lists all values of $R$ for the different figures in this section.
In Fig.~\ref{fig:result} we show a colour-coded tangent vector field together with the colour disk.

For simplicity reasons, for image functions as well as surfaces we plot their piecewise linear approximations.
Moreover, the visualised vector fields are evaluated at the centroids and result in piecewise constant colour-coded images.

\subsection{Results}

We performed several experiments on said zebrafish microscopy data.
In order to obtain an approximation of the evolving surface, we minimised functional~\eqref{eq:surffunctional} by solving the optimality conditions~\eqref{eq:surfoptcond}.
As mentioned in Sec.~\ref{sec:acquisition}, approximate cell centres serve as input.
The parameter of the Sobolev space $H^{s}(\mathcal{S})$ was chosen as $s = 3 + \epsilon$, where $\epsilon = 2.2204 \cdot 10^{-16}$ is the machine precision, cf. also the discussion regarding theoretical requirements in Sec.~\ref{sec:surface}.
The regularisation parameter was set to $\beta = 10^{-4}$ and the finite-dimensional subspace was chosen as
\begin{equation*}
	\mathcal{Q} = \mathrm{span} \Big\lbrace \tilde{Y}_{nj}: n = 0, \dots, 30, j = 1, \dots, 2n + 1 \Big\rbrace.
\end{equation*}

In the second step, we computed a minimiser of functional~\eqref{eq:functional} as outlined in Sec.~\ref{sec:numericalapprox}.
Here, the finite-dimensional subspace was chosen as
\begin{equation*}
	\mathcal{U} = \mathrm{span} \Big\lbrace \mathbf{\hat{y}}_{nj}^{(i)}: n = 1, \dots, 50, j = 1, \dots, 2n + 1, i = 2, 3 \Big\rbrace.
\end{equation*}

The linear systems resulting from optimality conditions~\eqref{eq:optcondpullback} and~\eqref{eq:surfoptcond} were solved by means of the General Minimal Residual Method (GMRES) using an Intel Xeon E5-1620 $3.6 \, \mathrm{GHz}$ workstation equipped with $128 \, \mathrm{GB}$ RAM.
Solutions to~\eqref{eq:optcondpullback} and~\eqref{eq:surfoptcond} converged within 1000 and 100 iterations, respectively, to a relative residual of $10^{-2}$.
The overall runtime was dominated by the evaluation of the integrals \eqref{eq:a}, \eqref{eq:d}, and \eqref{eq:b}.
In our Matlab implementation it amounts to several hours.
However, the resulting linear system can typically be solved within seconds.
Both implementation and data are available on our website.\footnote{\url{http://www.csc.univie.ac.at}}

Figures~\ref{fig:rho2} and~\ref{fig:rho3} portray a minimising function of $\mathcal{F}_{\beta}$ for frames 70 and 71 of the image sequence.
The resulting surface is depicted in Fig.~\ref{fig:data} and in Fig.~\ref{fig:data2}.
Clearly, it reflects the geometry appropriately and contains the desired cell features, cf. also the unprocessed microscopy data in Fig.~\ref{fig:raw}.

\begin{figure}[t]
	\includegraphics[width=0.49\textwidth]{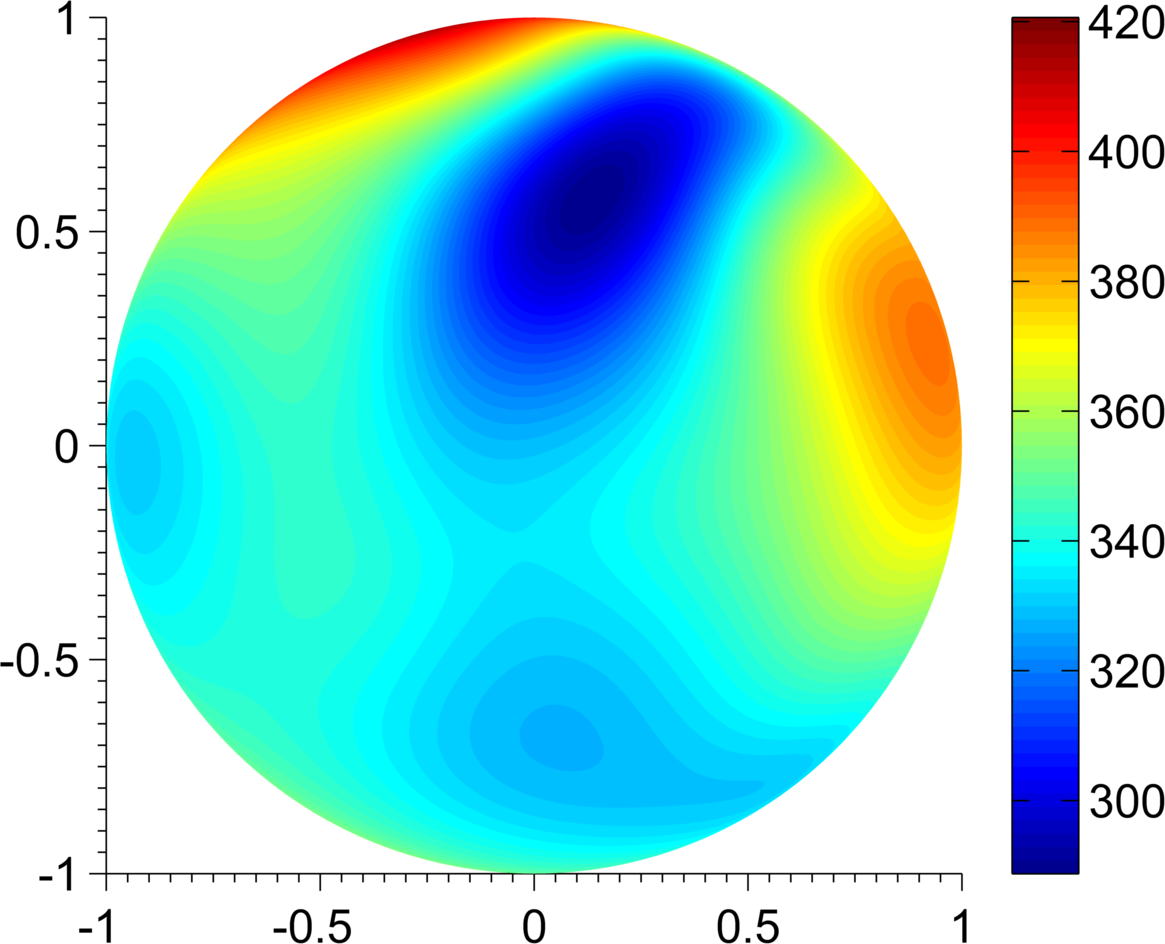} \hfill
	\includegraphics[width=0.49\textwidth]{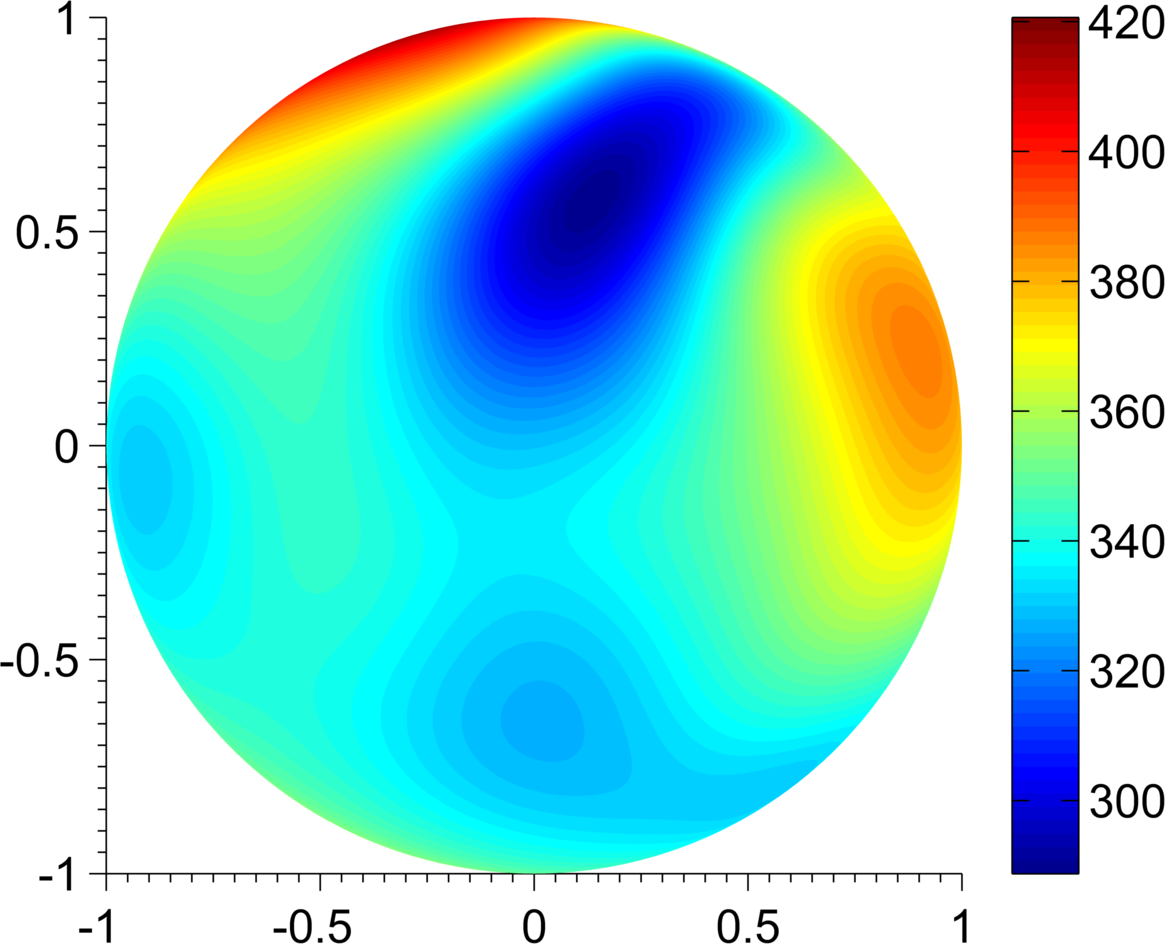}
	\caption{Function $\tilde{\rho}_{h}$ obtained by minimising $\mathcal{F}_{\beta}$ for frames 70 (left) and 71 (right). Colour corresponds to the radius of the fitted surface. $\mathcal{S}_{h}^{2}$ is depicted in a top view.}
	\label{fig:rho2}
\end{figure}

\begin{figure}[t]
	\includegraphics[width=0.49\textwidth]{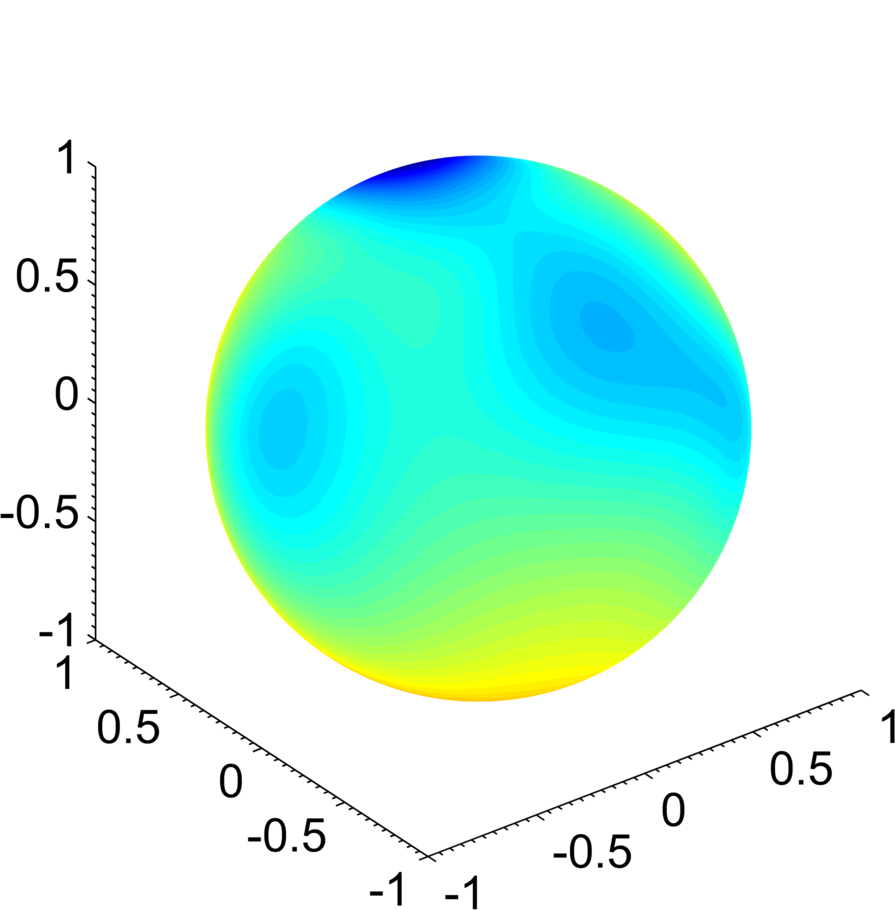} \hfill
	\includegraphics[width=0.49\textwidth]{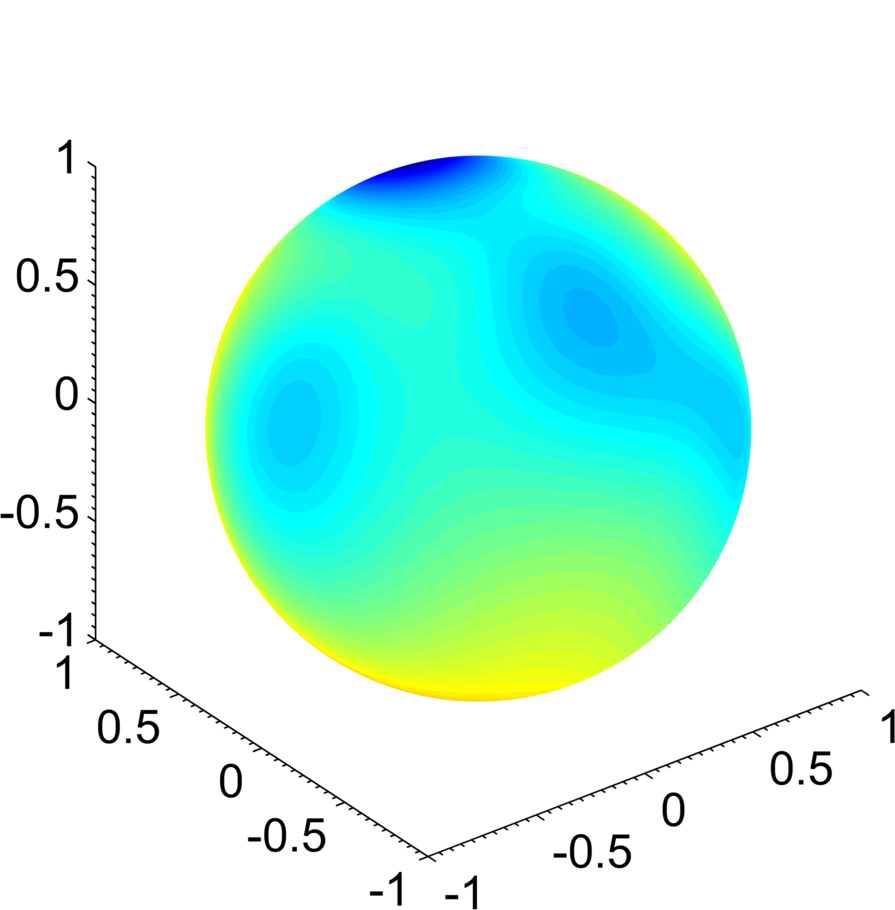} \\
	\includegraphics[width=0.49\textwidth]{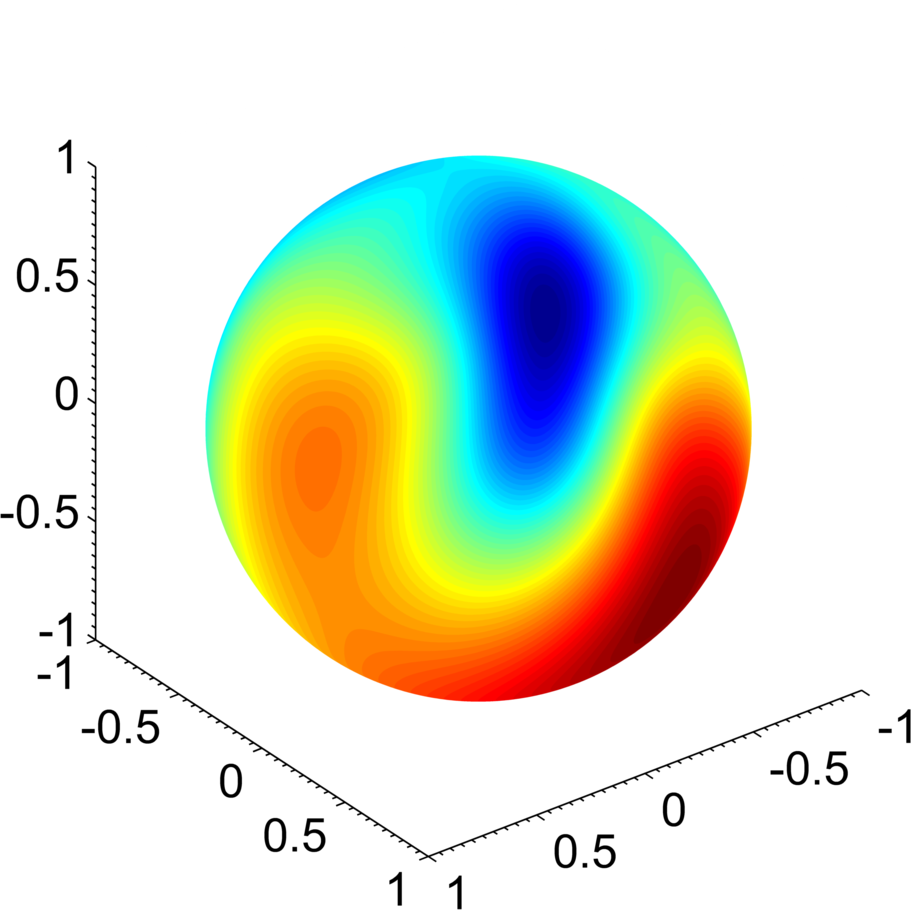} \hfill
	\includegraphics[width=0.49\textwidth]{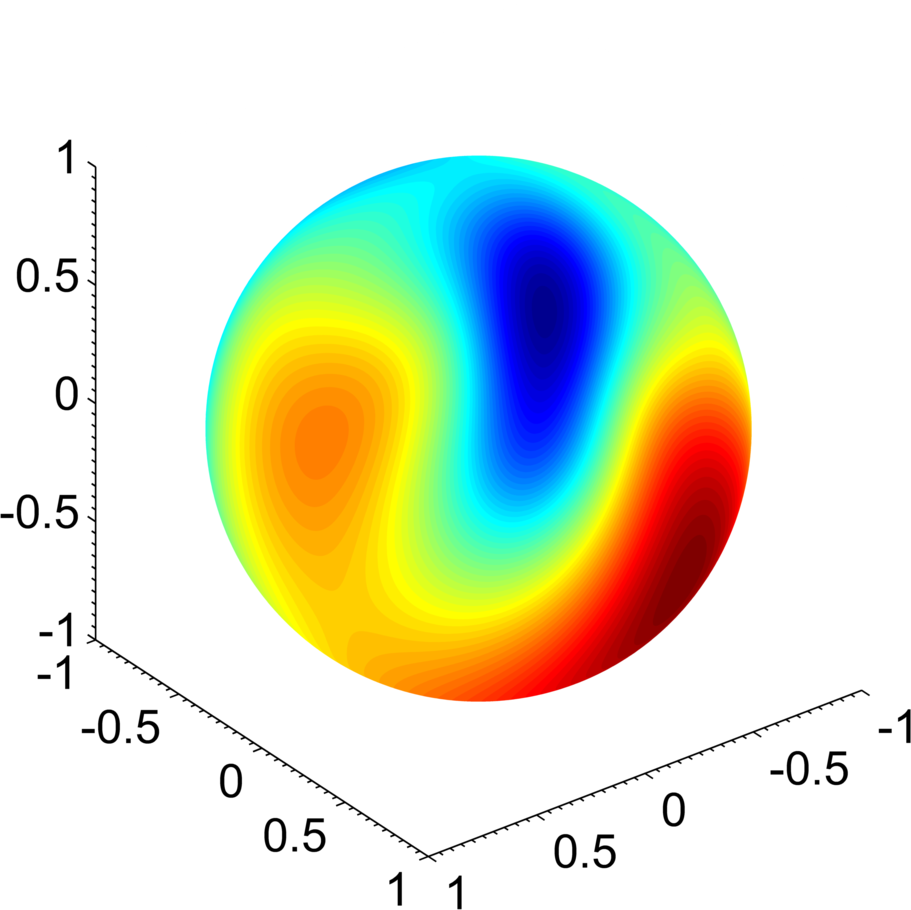}
	\caption{Illustration of the function $\tilde{\rho}_{h}$ for frames 70 (left column) and 71 (right column). The bottom row differs from the top row by a rotation of 180 degrees around the $x_{3}$-axis.}
	\label{fig:rho3}
\end{figure}

\begin{table}[t]
	\centering
	\begin{tabular}{@{}l@{\hspace{0.75em}}@{}l@{\hspace{0.75em}}@{}l@{\hspace{0.75em}}@{}l@{\hspace{0.75em}}@{}l@{\hspace{0.75em}}@{}l@{\hspace{0.75em}}@{}l@{\hspace{0.75em}}@{}l@{\hspace{0.75em}}@{}l@{\hspace{0.75em}}@{}l@{\hspace{0.75em}}}
		Figure & \ref{fig:result} & \ref{fig:flow3} (a) & \ref{fig:flow3} (b) & \ref{fig:flow3} (c) & \ref{fig:flow3} (d) & \ref{fig:flow2} (a) & \ref{fig:flow2} (b) & \ref{fig:flow2} (c) & \ref{fig:flow2} (d) \\
		\hline
		$R$ & $5.18$ & $9.86$ & $5.18$ & $3.64$ & $2.52$ & $9.86$ & $5.18$ & $3.64$ & $2.52$
	\end{tabular}
	\caption{Radii $R$ of the colour disks used for colour-coded visualisation of tangent vector fields.}
	\label{tab:radii}
\end{table}

In a second step we solved for minimisers of $\mathcal{E}_{\alpha}$ for different values of the regularisation parameter $\alpha$.
Figure~\ref{fig:result} depicts the optical flow field for $\alpha = 10^{-1}$.
The tangent vector field is visualised as discussed in Sec.~\ref{sec:visualisation}.
Note that in all figures the colour disk has been scaled for better illustration.
In Fig.~\ref{fig:flow3}, we illustrate tangent vector fields by means of the colour-coding obtained for $\alpha = 10^{-2}$, $\alpha = 10^{-1}$, $\alpha = 1$, and $\alpha = 10$.
Finally, in Fig.~\ref{fig:flow2} we show the same results but in a top view.

\begin{figure}[t]
	\includegraphics[width=0.49\textwidth]{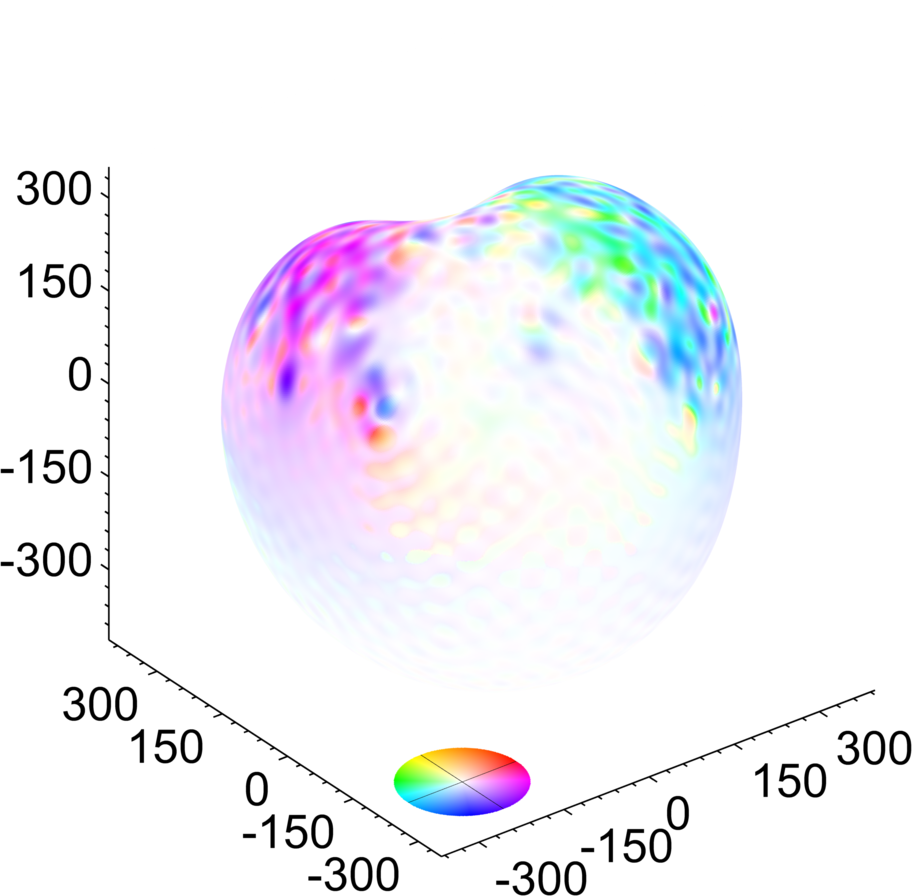} \hfill
	\includegraphics[width=0.49\textwidth]{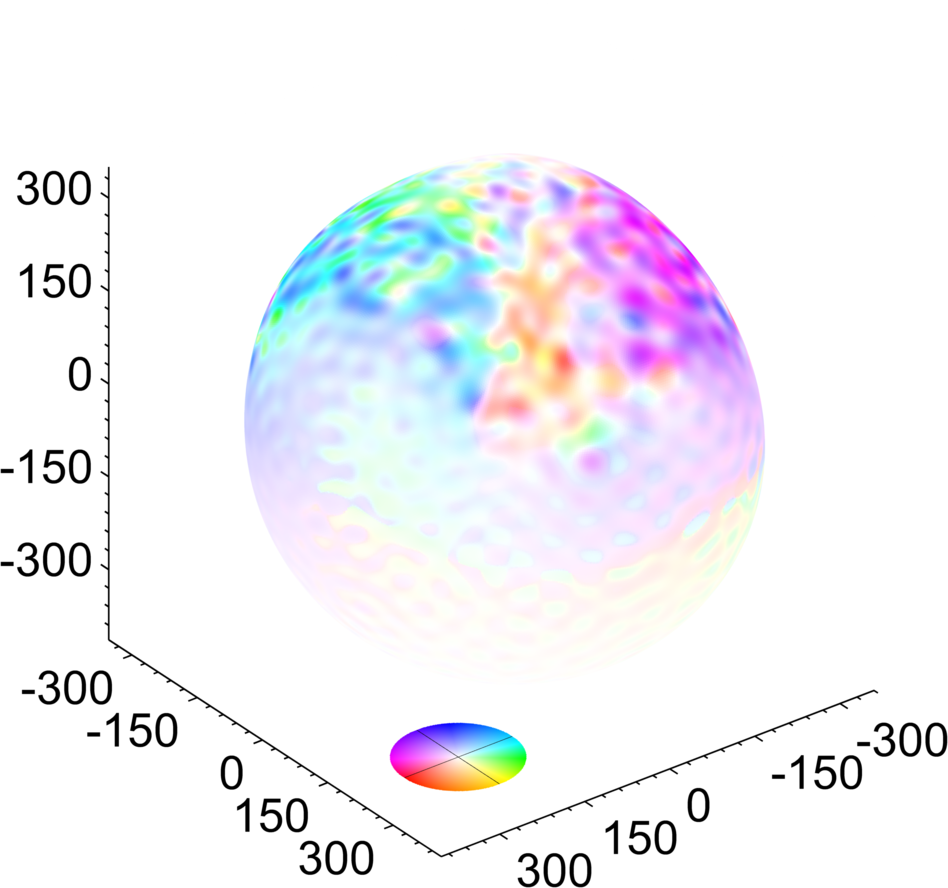}
	\caption{Tangent vector field minimising $\mathcal{E}_{\alpha}$. Depicted is the colour-coded optical flow field computed between frames 70 and 71. The right image differs from the left by a rotation of 180 degrees around the $x_{3}$-axis.}
	\label{fig:result}
\end{figure}

\begin{figure}[t]
	\includegraphics[width=0.24\textwidth]{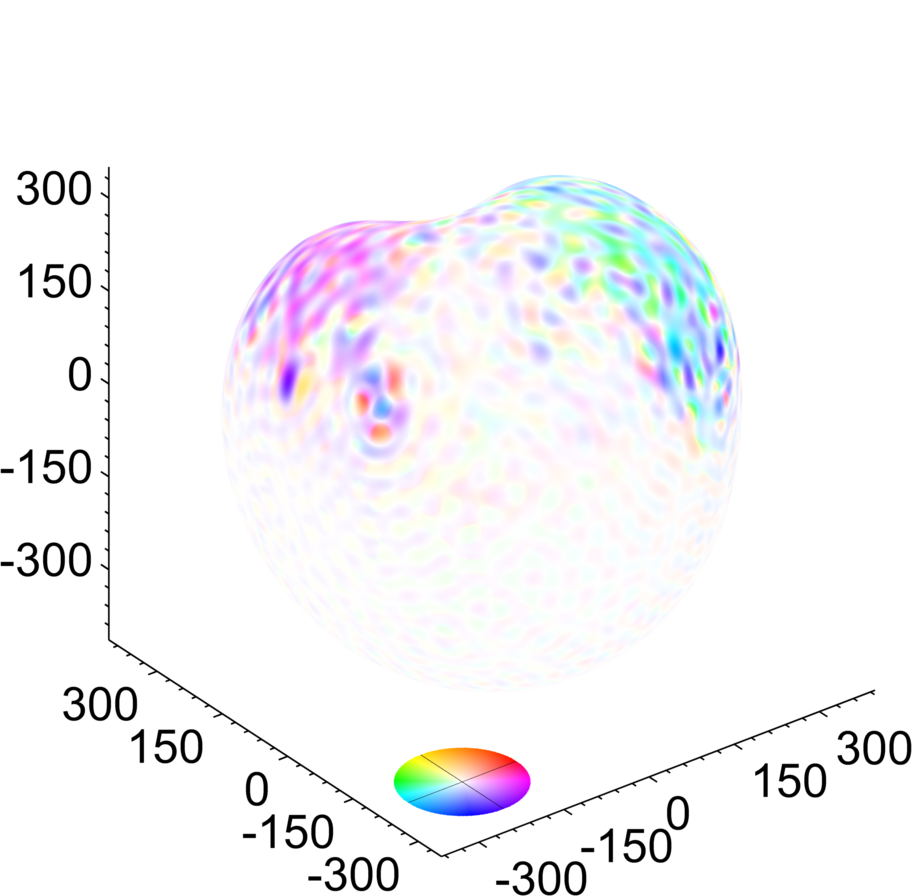} \hfill
	\includegraphics[width=0.24\textwidth]{figures/flow3/flow3-frames-140-142-unfiltered-1-50-7-3-600dpi} \hfill
	\includegraphics[width=0.24\textwidth]{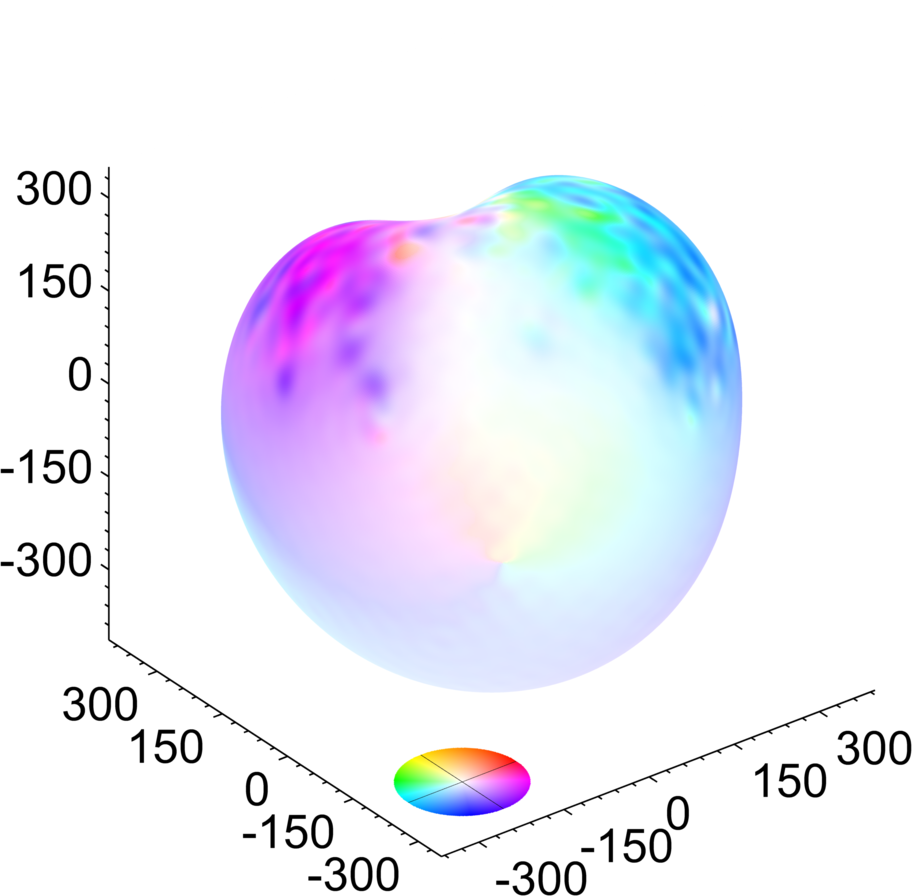} \hfill
	\includegraphics[width=0.24\textwidth]{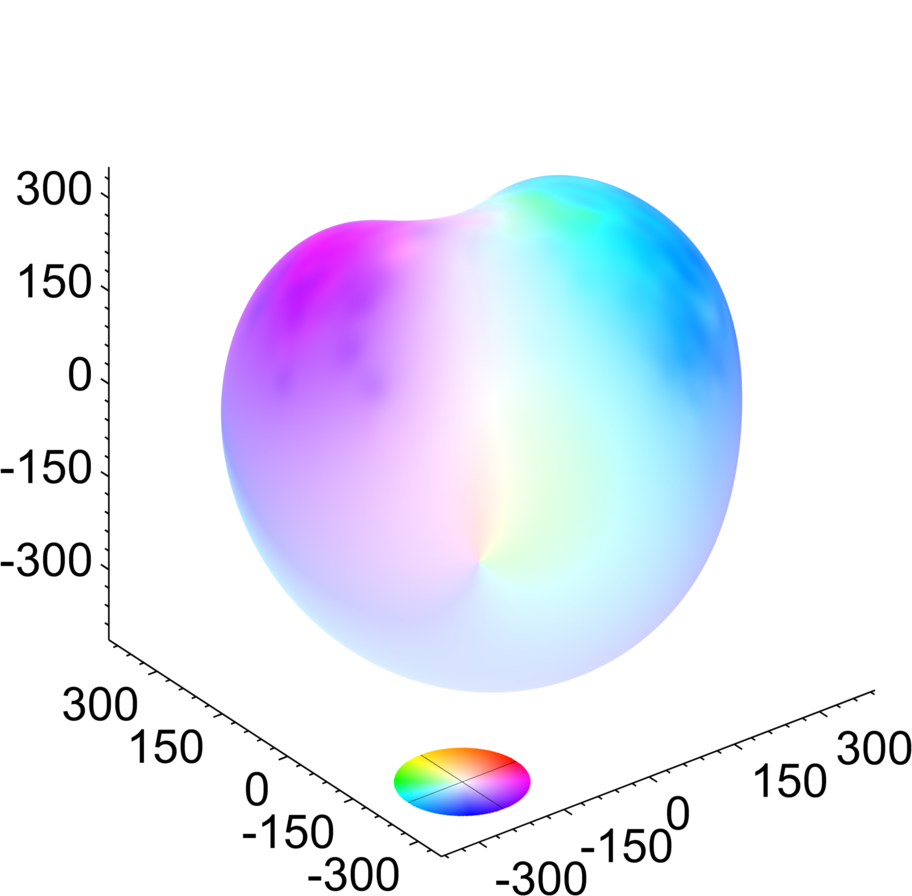} \\			
	\includegraphics[width=0.24\textwidth]{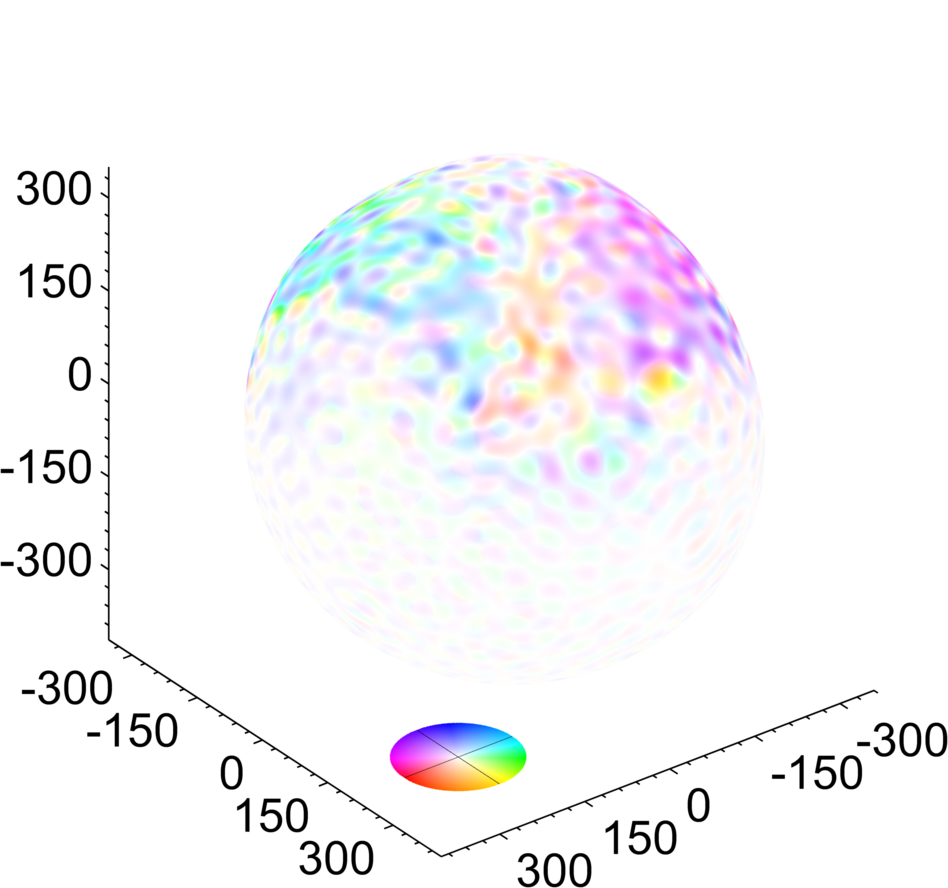} \hfill
	\includegraphics[width=0.24\textwidth]{figures/flow3/flow3-frames-140-142-unfiltered-1-50-7-3-rotated-600dpi} \hfill
	\includegraphics[width=0.24\textwidth]{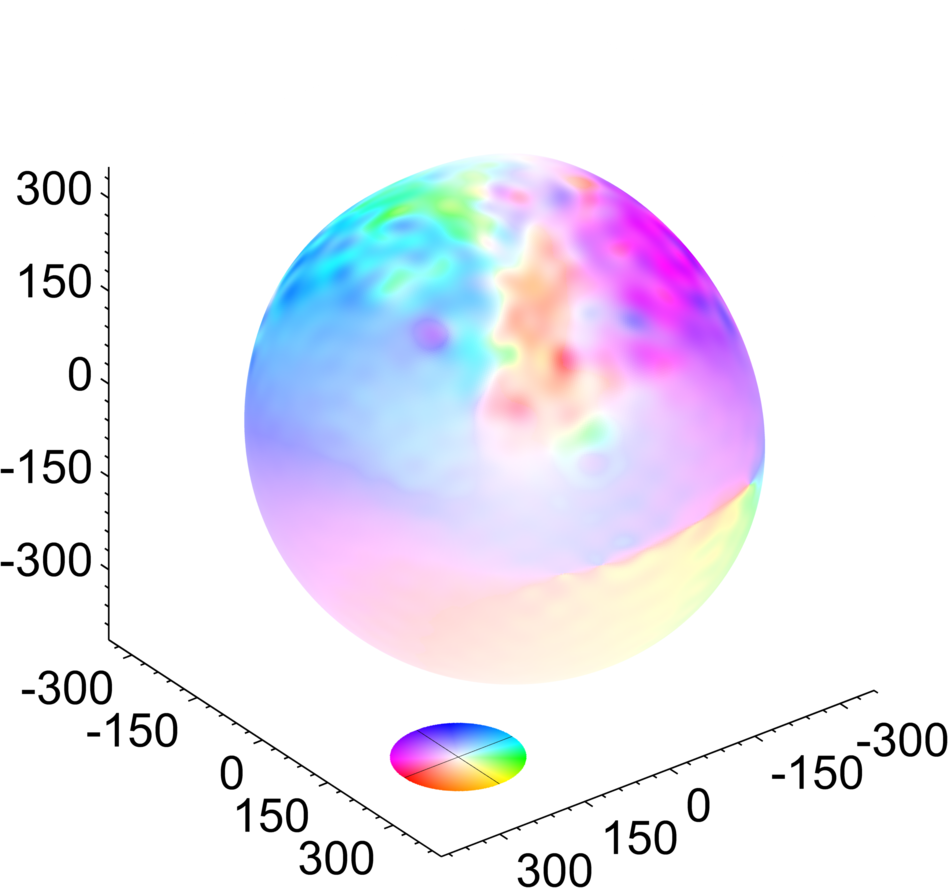} \hfill
	\includegraphics[width=0.24\textwidth]{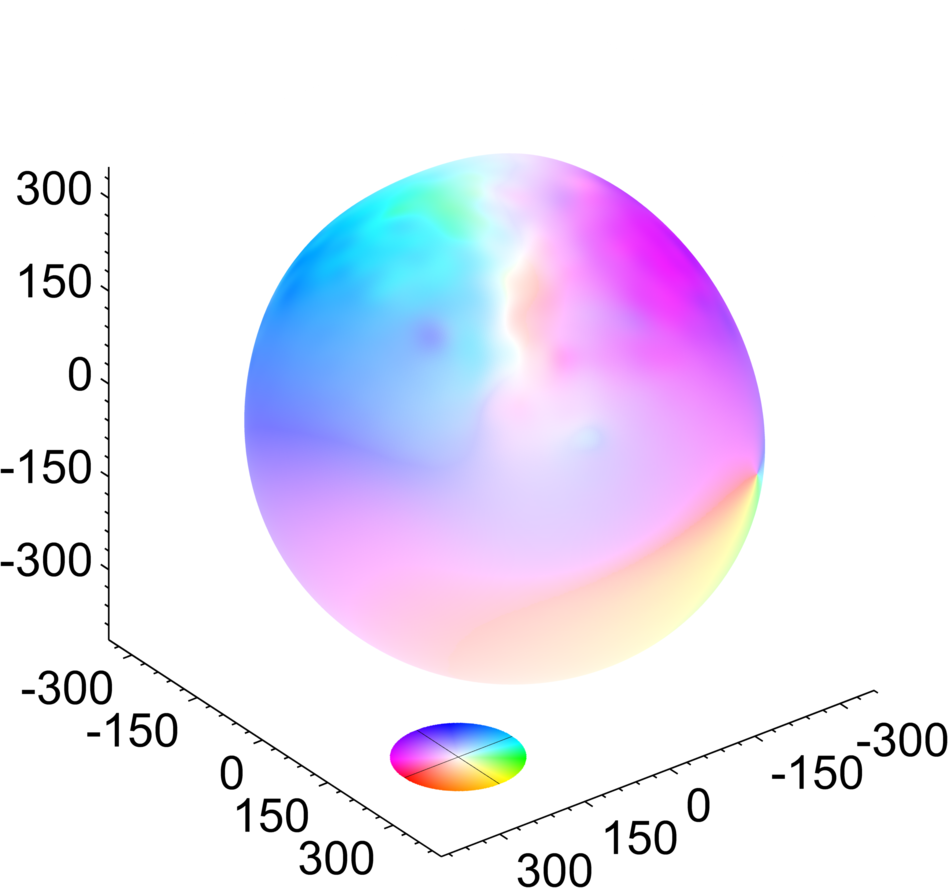}
	\caption{Visualisation of the optical flow field obtained for different values of $\alpha$. The bottom row differs from the top view by a rotation of 180 degrees around the $x_{3}$-axis. From left to right: a) $\alpha = 10^{-2}$, b) $\alpha = 10^{-1}$, c) $\alpha = 1$, and d) $\alpha = 10$.}
	\label{fig:flow3}
\end{figure}

\begin{figure}[t]
	\includegraphics[width=0.49\textwidth]{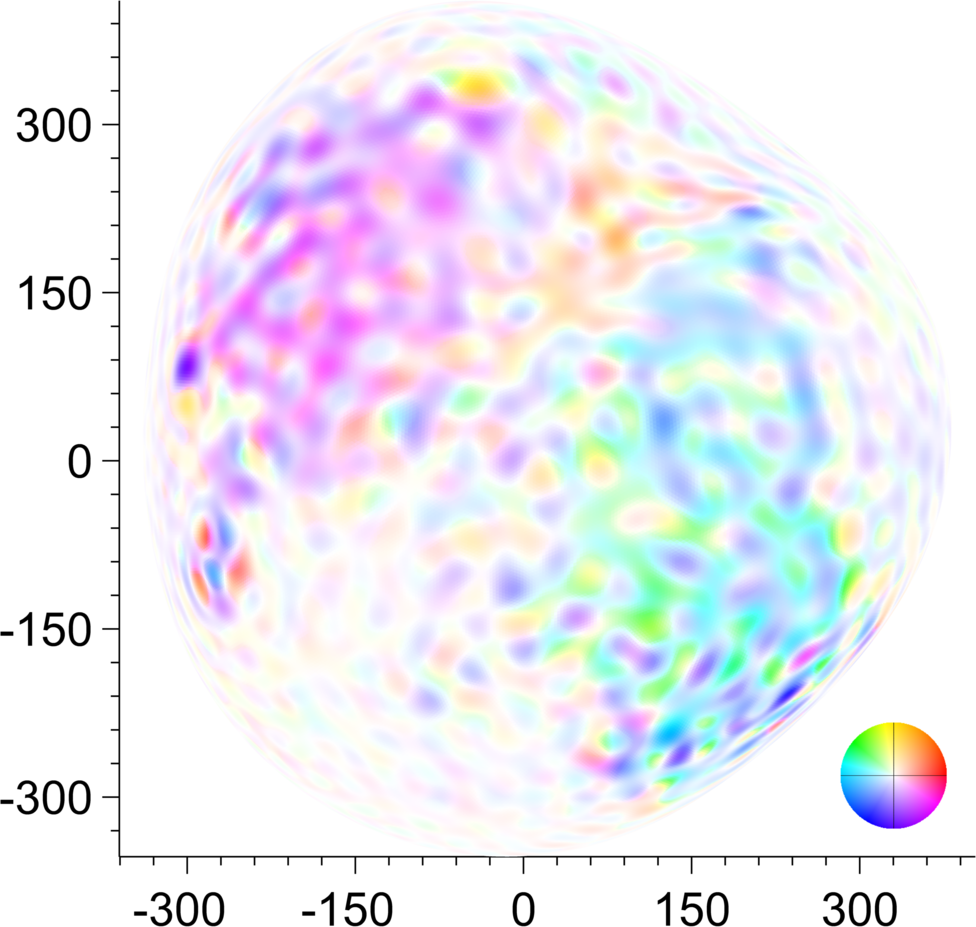} \hfill
	\includegraphics[width=0.49\textwidth]{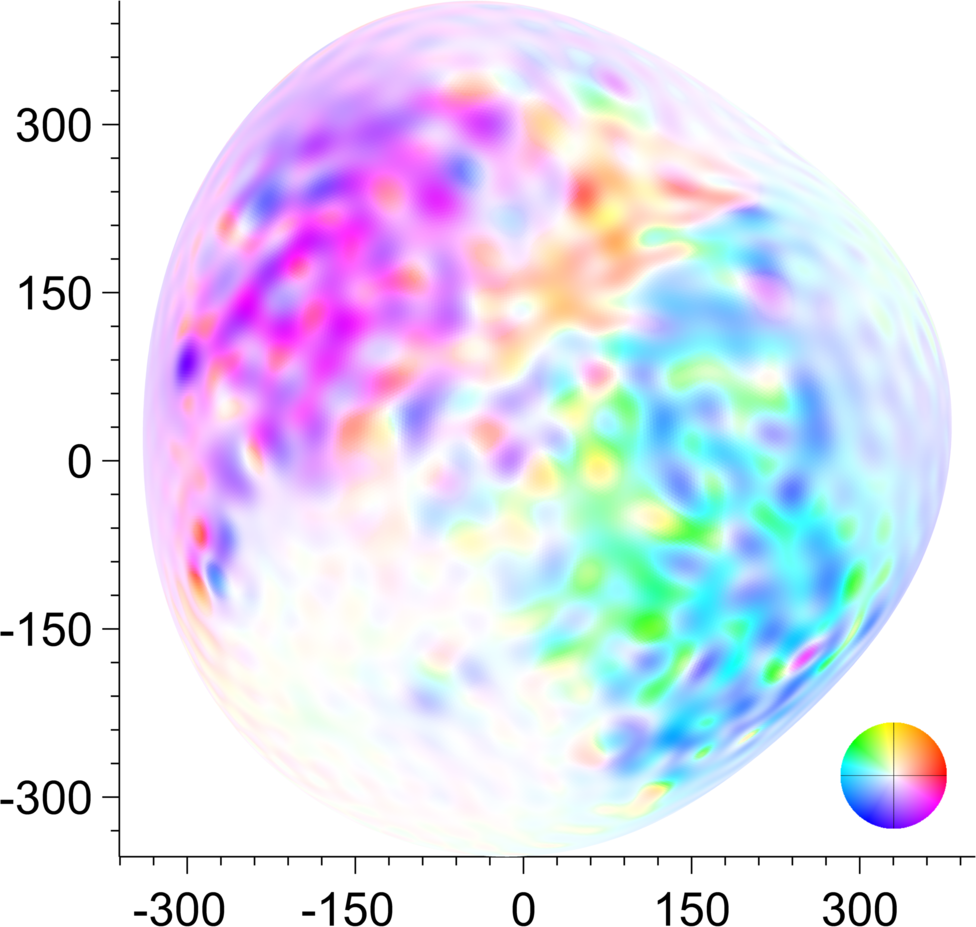} \\
	\includegraphics[width=0.49\textwidth]{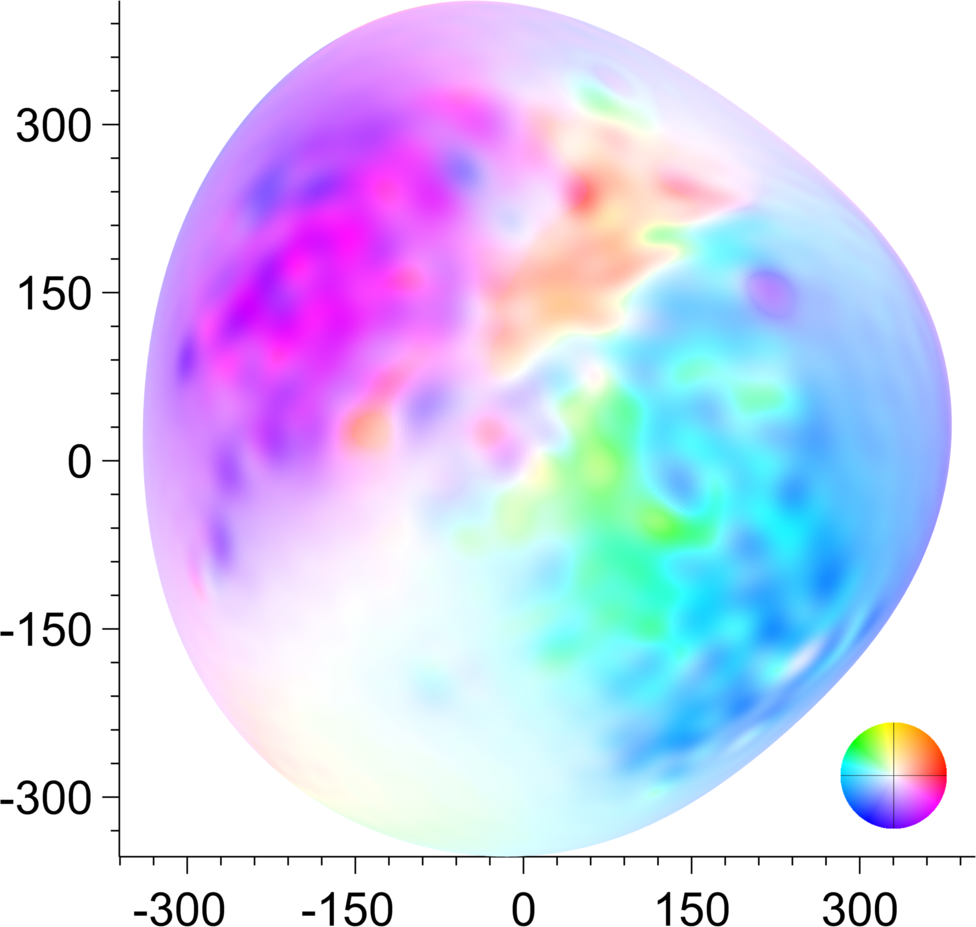} \hfill
	\includegraphics[width=0.49\textwidth]{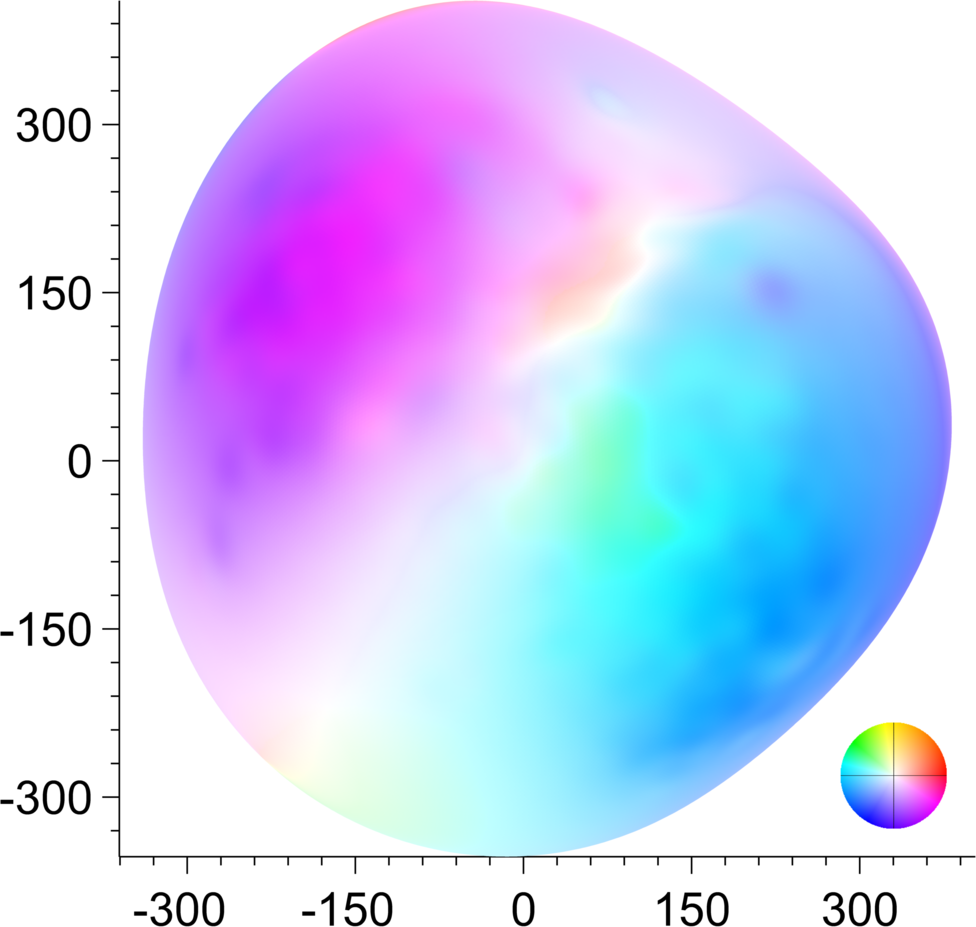}
	\caption{Top view of the optical flow field computed for different values of $\alpha$. From left to right, top to bottom: a) $\alpha = 10^{-2}$, b) $\alpha = 10^{-1}$, c) $\alpha = 1$, and d) $\alpha = 10$.}
	\label{fig:flow2}
\end{figure}
\section{Conclusion}

With the goal of efficient cell motion analysis we considered optical flow on evolving surfaces.
As a prototypical example we restricted ourselves to surfaces parametrised from the round sphere and showed that 4D microscopy data of a living zebrafish embryo can be faithfully represented in this way.
In contrast to previous works, where only a section of the embryo or a spherical approximation was considered, our approach fully attributes the geometry and models the embryo as as closed surface of genus zero.
The resulting energy functional was solved by means of a Galerkin method based on vector spherical harmonics.
Moreover, the parametrisation of the moving sphere-like surface was obtained from the data by solving a surface interpolation problem.
Scalar spherical harmonics expansion allows to easily meet the smoothness requirements of the surface.
Finally, we conducted several experiments based on said microscopy data.
Our results show that cell motion can be indicated reasonably well by the proposed approach.

\paragraph{Acknowledgements}
We thank Pia Aanstad from the University of Innsbruck for sharing her biological insight and for kindly providing the microscopy data.
Moreover, we are grateful to Peter Elbau and Clemens Kirisits for their helpful comments, and to Jos\'{e}~A. Iglesias for carefully proofreading an earlier version of this article and providing valuable feedback.
This work has been supported by the Vienna Graduate School in Computational Science (IK I059-N) funded by the University of Vienna.
In addition, we acknowledge the support by the Austrian Science Fund (FWF) within the national research network ``Geometry + Simulation" (project S11704, Variational Methods for Imaging on Manifolds).

\appendix
\section*{Appendix} \label{sec:appendix}

It remains to give the calculations regarding the piecewise linear approximations of vector spherical harmonics on $\mathcal{S}_{h}^{2}$.
Both equations in Prop.~\ref{def:vspharm} follow directly by expanding the definitions of the tangential vector spherical harmonics~\eqref{eq:fullynormalisedvspharm}.
For the fist identity, that is~\eqref{eq:vspharm1}, we have
\begin{equation*}
	\mathbf{\tilde{y}}_{h}^{(2)} = \lambda_{n}^{-1/2} \nabla_{\mathcal{S}_{h}^{2}} \tilde{Y}_{h} = \lambda_{n}^{-1/2} \sum_{j = 1}^{N_{h}} \tilde{\bar{Y}}(v_{j}) \nabla_{\mathcal{S}_{h}^{2}} \tilde{\varphi}_{j}.
\end{equation*}

The second identity, that is~\eqref{eq:vspharm2}, follows by the fact that
\begin{equation*}
	2 \abs{T_{i}} = \abs{\partial_{1} \x_{i} \times \partial_{2} \x_{i}}, \quad \mathbf{\tilde{N}}_{i} = \frac{\partial_{1} \x_{i} \times \partial_{2} \x_{i}}{\abs{\partial_{1} \x_{i} \times \partial_{2} \x_{i}}},
\end{equation*}
and by application of the vector triple product rule, yielding
\begin{align*}
	\mathbf{\tilde{y}}_{h}^{(3)} & = \lambda_{n}^{-1/2} \nabla_{\mathcal{S}_{h}^{2}} \tilde{Y}_{h} \times \mathbf{\tilde{N}}_{i} \\
	& = \lambda_{n}^{-1/2} \nabla_{\mathcal{S}_{h}^{2}} \tilde{Y}_{h} \times \frac{\partial_{1} \x_{i} \times \partial_{2} \x_{i}}{\abs{\partial_{1} \x_{i} \times \partial_{2} \x_{i}}} \\
	& = \frac{\lambda_{n}^{-1/2}}{2 \abs{T_{i}}} \bigl( (\nabla_{\mathcal{S}_{h}^{2}} \tilde{Y}_{h} \cdot \partial_{2} \x_{i}) \partial_{1} \x_{i} - (\nabla_{\mathcal{S}_{h}^{2}} \tilde{Y}_{h} \cdot \partial_{1} \x_{i}) \partial_{2} \x_{i} \bigr),
\end{align*}
where the last equality results from the definition of the interpolation~\eqref{eq:spharminterp} of $\tilde{Y}_{h}$ and the directional derivative~\eqref{eq:partialf} on the triangular face, that is
\begin{equation*}
	\nabla_{\mathcal{S}_{h}^{2}} \tilde{Y}_{h} \cdot \partial_{k} \x_{i} = \sum_{j = 1}^{N_{h}} \tilde{\bar{Y}}(v_{j}) \partial_{k} \varphi_{j}.
\end{equation*}

\def\cprime{$'$}
  \providecommand{\noopsort}[1]{}\def\ocirc#1{\ifmmode\setbox0=\hbox{$#1$}\dimen0=\ht0
  \advance\dimen0 by1pt\rlap{\hbox to\wd0{\hss\raise\dimen0
  \hbox{\hskip.2em$\scriptscriptstyle\circ$}\hss}}#1\else {\accent"17 #1}\fi}

\end{document}